\documentclass[11pt]{article}
\usepackage{fullpage}
\usepackage{amssymb,amsmath,amsthm}
\usepackage{graphicx}
\usepackage{url}
\usepackage{color}
\usepackage{here}

\newtheorem{theorem}{Theorem}
\newtheorem{lemma}[theorem]{Lemma}
\newtheorem{propo}[theorem]{Proposition}

\newtheorem{clm}[theorem]{Claim}

\newcommand{\diam}{\ensuremath{\mathrm{diam}}}
\newcommand{\slope}{\ensuremath{\mathrm{slope}}}
\newcommand{\conv}{\ensuremath{\mathrm{conv}}}
\newcommand{\width}{\ensuremath{\mathrm{width}}}

\newcommand{\starify}{\ensuremath{\texttt{starify}}}
\newcommand{\cx}{\ensuremath{\mathrm{cx}}}
\newcommand{\op}{\ensuremath{\mathbf{op}}}
\newcommand{\ex}{\ensuremath{\mathbf{ex}}}
\newcommand{\ce}{\ensuremath{\mathbf{ce}}}
\newcommand{\ro}{\ensuremath{\mathbf{ro}}}
\newcommand{\es}{\ensuremath{\mathbf{es}}}
\newcommand{\er}{\ensuremath{\mathbf{er}}}
\newcommand{\ses}{\ensuremath{\mathbf{ses}}}
\newcommand{\sro}{\ensuremath{\mathbf{sro}}}
\newcommand{\ser}{\ensuremath{\mathbf{ser}}}
\newcommand{\sop}{\ensuremath{\mathbf{sop}}}
\newcommand{\sce}{\ensuremath{\mathbf{sce}}}
\graphicspath{{figures/}}

\newcommand{\ceil}[1]{\ensuremath{\left \lceil #1 \right \rceil}}

\makeatletter
\g@addto@macro\bfseries{\boldmath}
\makeatother

\usepackage[hidelinks]{hyperref}

\begin{document}

\title{Transition Operations over Plane Trees\thanks{A preliminary version of this paper appeared in the \emph{Proceedings of the 13th Latin American Theoretical INformatics Symposium (LATIN)}, 2018, pp. 835--848, LNCS, Springer,
\href{https://doi.org/10.1007/978-3-319-77404-6_60}{doi:10.1007/978-3-319-77404-6\_60}.}}

\author{Torrie L. Nichols\thanks{Department of Mathematics, California State University Northridge, Los Angeles, CA, USA.
\texttt{torrie.nichols.643@my.csun.edu, csaba.toth@csun.edu}}
\and
Alexander Pilz\thanks{Institute of Software Technology, Graz University of Technology, Austria.
\texttt{apilz@ist.tugraz.at}}
\and
Csaba D. T\'oth\footnotemark[2]
\and
Ahad N. Zehmakan\thanks{Department of Computer Science, ETH Z\"urich, Switzerland.
\texttt{abdolahad.noori@inf.ethz.ch}}
}

\date{}

\maketitle

\begin{abstract}
The operation of transforming one spanning tree into another by replacing an edge has been considered widely, both for general and planar straight-line graphs.
For the latter, several variants have been studied (e.g., \emph{edge slides} and \emph{edge rotations}).
In a transition graph on the set $\mathcal{T}(S)$ of noncrossing straight-line spanning trees on a finite point set $S$ in the plane, two spanning trees are connected by an edge if one can be transformed into the other by such an operation.
We study bounds on the diameter of these graphs, and consider the various operations on point sets in both general position and convex position.
In addition, we address variants of the problem where operations may be performed simultaneously or the edges are labeled.
We prove new lower and upper bounds for the diameters of the corresponding transition graphs
and pose open problems.\\

\end{abstract}

\section{Introduction}
\label{sec:intro}

For a set $S$ of $n$ points in the plane,
let $\mathcal{T}(S)$ denote the set of noncrossing straight-line spanning trees on the vertex set $S$.
In the last 20 years, five different operations have been introduced over $\mathcal{T}(S)$. While all five operations are based on a classic exchange property of graphic matroids~\cite{Oaxley1993}, geometric conditions yield a rich hierarchy.

\smallskip\noindent\textbf{Elementary Operations.}
Let $T_1=(S,E_1)$ and $T_2=(S,E_2)$ be two trees in $\mathcal{T}(S)$.
\begin{itemize}\itemsep 0pt
\item An \textbf{exchange} is an operation that replaces $T_1$ by $T_2$ so that there are two edges, $e_1$ and $e_2$, such that $E_1\setminus E_2=\{e_1\}$ and $E_2\setminus E_1=\{e_2\}$ (i.e., $E_2$ can be obtained from $E_1$ by deleting an edge $e_1\in E_1$ and inserting a new edge $e_2\in E_2$).
\item A \textbf{compatible exchange} is an exchange such that the graph $(S,E_1\cup E_2)$ is
    a noncrossing straight-line graph (i.e., $e_1$ and $e_2$ do not cross).
\item A \textbf{rotation} is a compatible exchange such that $e_1$ and $e_2$ have a common endpoint $p=e_1\cap e_2$.
\item An \textbf{empty-triangle rotation} is a rotation such that the edges of neither $T_1$ nor $T_2$ intersect the interior of the triangle $\Delta(pqr)$ formed by the vertices of $e_1$ and~$e_2$.
\item An \textbf{edge slide} is an empty-triangle rotation such that $qr\in E_1\cap E_2$.
\end{itemize}

See \figurename~\ref{fig:operations} for illustrations. All five operations that we consider have been defined prior to our work (see below), but this is the first comprehensive study of all five operations.
Note that, for each of the five operations, the inverse of an operation (i.e., transforming $T_2$ into $T_1$) is of the same type.
Each operation $\op$ defines an undirected \emph{transition graph} $\mathcal{G}_{\op}(S)$, whose vertex set is $\mathcal{T}(S)$, and there is an edge between two trees $T_1,T_2\in \mathcal{T}(S)$ if and only if an operation $\op$ can transform $T_1$ into $T_2$. The transition graphs for all five operations are known to be connected (see Section~\ref{ssec:previous}). The \emph{diameter} $\diam(\mathcal{G}_{\op}(S))$ of the transition graph $\mathcal{G}_{\op}(S)$ is thus the maximum length of a shortest sequence of operations $\op$ that transforms one noncrossing straight-line spanning tree in $\mathcal{T}(S)$ into another.
We are interested in the asymptotic growth rate of the function
$f_{\op}(n):=\max_{|S|=n}\diam(\mathcal{G}_{\op}(S))$.

\begin{figure}[t]
\centering
\includegraphics{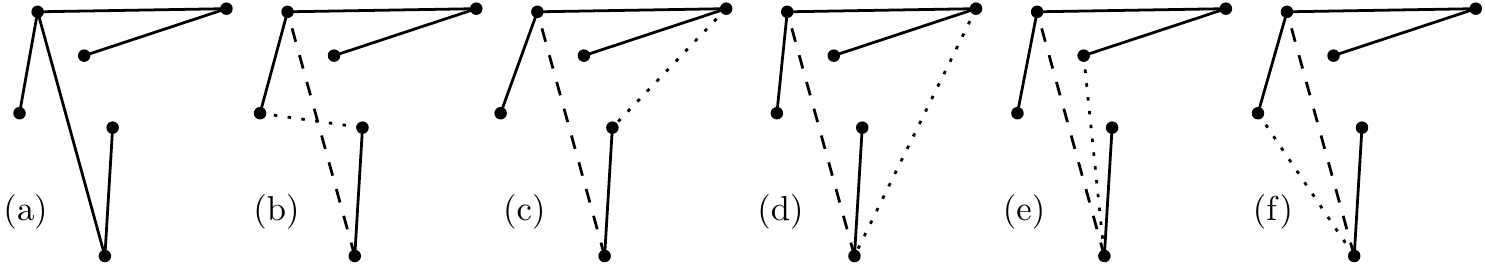}
\caption{A straight-line spanning tree (a), in which we replace the dashed edge by a dotted one, using an exchange (b), a compatible exchange (c), a rotation (d), an empty-triangle rotation (e), and an edge slide (f).}
\label{fig:operations}
\end{figure}

\paragraph{Simultaneous Operations.}
For each elementary operation $\op$, we define a \emph{simultaneous operation} $\sop$ on $\mathcal{T}(S)$ as follows. For two trees $T_1=(S,E_1)$ and $T_2=(S,E_2)$ in $\mathcal{T}(S)$, the operation $\sop$ replaces $T_1$ by $T_2$ if there is a bijection between $E_1\setminus E_2$ and $E_2\setminus E_1$ (the \emph{old} edges and \emph{new} edges, resp.) and each pair $(e_1,e_2)\in (E_1\setminus E_2)\times (E_2\setminus E_1)$ of corresponding edges satisfies the geometric conditions of the elementary operation $\op$ on $T_1$.
Importantly, we do \emph{not} require the graph $(S,E_1\setminus \{e_1\}\cup\{e_2\})$ to be in $\mathcal{T}(S)$, it is sufficient that each pair $(e_1,e_2)$ satisfies the geometric conditions and $T_1,T_2\in \mathcal{T}(S)$.
In particular, there is no geometric condition for a \textbf{simultaneous exchange}. For
\textbf{simultaneous compatible exchange}, $(S,E_1\cup \{e_2\})$ is a noncrossing straight-line graph for every $e_2\in E_2\setminus E_1$, and $(S,\{e_1\}\cup  E_2)$ is a noncrossing straight-line graph for every $e_1\in E_1\setminus E_2$; consequently $(S,E_1\cup E_2)$ must be a noncrossing straight-line graph. Simultaneous rotation, empty-triangle-rotation, and edge slide pose additional conditions on the pairs of corresponding edges.
We define the graph $\mathcal{G}_{\sop}(S)$ and maximum diameter $f_{\sop}(n)$ for simultaneous operations analogously. Clearly, $f_{\sop}(n)\leq f_{\op}(n)$.

\paragraph{General Position and Convex Position.}
We assume that $S$ is in \emph{general position} (i.e., no three points in $S$ are collinear). This assumption is for convenience only (all diameter bounds would hold regardless but would require a detailed discussion of special cases). Previous results (cf.~Section~\ref{ssec:previous}) are also subject to this assumption.
Arguably the most important special case is that $S$ is in convex position.
We are also interested in the asymptotic growth rate of the function $f_{\op}^{\cx}(n)$, which is equal to
$\max_{|S|=n}\diam(\mathcal{G}_{\op}(S))$, where $S$ is in convex position.
(Observe that, for the operations mentioned, the graphs $\mathcal{G}_{\op}(S)$ are isomorphic for any two sets $S$ of $n$ points in convex position.)
The function $f_{\sop}^{\cx}(n)$ is defined analogously.
Trivially, $f_{\op}^{\cx}(n)\leq f_{\op}(n)$ and $f_{\sop}^{\cx}(n)\leq f_{\sop}(n)$ for any operation~$\op$.

\paragraph{Labeled Edges.}
Each of the five elementary operations defined above exchanges an edge of a spanning tree with a new edge, and the simultaneous operations require a bijection between the old and the new edges.
We can extend these operations to edge-labeled spanning trees such that whenever an old edge $e_1$ is replaced by a new edge $e_2$, the label of $e_1$ is transferred to $e_2$.
For a spanning tree $T=(S,E)$ on a set $S$ of $n$ points in general position, an \emph{edge labeling} is a bijective function $\lambda:E\rightarrow \{1,\ldots , n-1\}$.
In particular, every tree in $\mathcal{T}(S)$ admits $(n-1)!$ edge labelings.
Denoting by $\mathcal{L}(S)$ the set of
edge-labeled noncrossing straight-line spanning trees on $S$, we can define a transition graph $\mathcal{G}_{\op}^{L}(S)$ on the vertex set $\mathcal{L}(S)$ in which two edge-labeled trees are adjacent
if and only if an operation $\op$ can transform one into the other.
By definition, $\mathcal{G}_{\op}(S)$ is a quotient graph of $\mathcal{G}_{\op}^{L}(S)$; consequently
$\diam(\mathcal{G}_{\op}(S))\leq \diam(\mathcal{G}_{\op}^L(S))$.

\paragraph{Organization.}
We summarize the current best lower and upper bounds for the diameter of transition graphs under the five elementary operations in Section~\ref{ssec:results}.
To put our results into context, we review related previous work on other elementary graph operations in Section~\ref{ssec:previous}.
Our new results on the diameter of transition graphs under rotation, empty-triangle rotation, and edge slide are presented in Sections~\ref{sec:Rotation}--\ref{sec:EdgeSlide}.
We consider the edge-labeled variant of the problem in Section~\ref{sec:labeled}, and conclude with open problems in Section~\ref{sec:conclusion}.

\subsection{Contributions and Related Previous Results}
\label{ssec:results}

The current best diameter bounds for the five operations and
their simultaneous variants are summarized in Table~\ref{tbl:bounds_general}.
Bounds for points in convex position are shown in Table~\ref{tbl:bounds_convex}.
The operations are presented from strongest to weakest: we say that an operation $\op_1$ is \emph{stronger} than operation $\op_2$ if every operation $\op_2$ is also an operation $\op_1$.
As $\mathcal{G}_{\op_2}(S)$ is a subgraph of $\mathcal{G}_{\op_1}(S)$, we have $f_{\op_1}(n)\leq f_{\op_2}(n)$ and $f_{\op_1}^{\cx}(n)\leq f_{\op_2}^{\cx}(n)$.
It is worth noting that even though we briefly review the current best bounds for the two strongest operations, our main results concern the three weakest operations: rotation (Section~\ref{sec:Rotation}), empty-triangle rotation (Section~\ref{sec:EmptyTriangleRotation}), and edge slide (Section~\ref{sec:EdgeSlide}).
See Tables~\ref{tbl:bounds_general} and~\ref{tbl:bounds_convex}, where our contributions are marked with the corresponding theorems and propositions.

\begin{table}[H]
\begin{center}
\bgroup
\def\arraystretch{1.15}
\begin{tabular}{|l|l|l|l|l|}
\hline
Operation& Single Oper. & Single Oper. & Simultaneous & Simultaneous  \\
& Lower Bd.  & Upper Bd. & Lower Bd. & Upper Bd.\\
\hline \hline
Exchange   & $\lfloor \frac{3n}{2}\rfloor -5$~\cite{HERNANDO1999} & $2n-4$ & 1 & 1 \\
\hline
Compatible Ex.  & $\lfloor \frac{3n}{2}\rfloor -5$ & $2n-4$ & $\Omega(\frac{\log n}{\log \log n})$ \cite{buchin2009transforming} & $O(\log n)$~\cite{aichholzer2002sequences}\\
\hline
Rotation  & $\lfloor \frac{3n}{2}\rfloor -5$ & $2n-4$~\cite{avis1996reverse} & $\Omega(\frac{\log n}{\log \log n})$ & $O(\log n)$ [Thm.~\ref{thm:rotations}]\\
\hline
Empty-Tri.~Rot. & $\lfloor \frac{3n}{2}\rfloor -5$ & $O(n\log n)$~[Thm.~\ref{thm:empty_rotation_general}] & $\Omega(\log n)$ [Thm.~\ref{thm:sim_empty_lower}] & $8n$ [Thm.~\ref{thm:sim_empty_upper}]\\
\hline
Edge Slide  & $\Omega(n^{2})$~\cite{aichholzer2007quadratic} & $O(n^{2})$ \cite{aichholzer2007quadratic} & $\Omega(n)$ [Prop.~\ref{pro:sim-slide}]& $O(n^{2})$  \cite{aichholzer2007quadratic}\\
\hline
\end{tabular}
\egroup
\end{center}
\caption{Diameter bounds for $n$ points in general position.}\label{tbl:bounds_general}
\end{table}

\begin{table}[H]
\begin{center}
\bgroup
\def\arraystretch{1.15}
\begin{tabular}{|l|l|l|l|l|}
\hline
Operation&  Single Oper. & Single Oper. & Simultaneous & Simultaneous  \\
 & Lower Bd.  & Upper Bd. & Lower Bd. & Upper Bd.\\
\hline \hline
Exchange  & $\lfloor \frac{3n}{2}\rfloor -5$~\cite{HERNANDO1999} & $2n-5$ & 1 & 1\\
\hline
Compatible Ex. &  $\lfloor \frac{3n}{2}\rfloor -5$ & $2n-5$ & 2 & 2 \\
\hline
Rotation  &  $\lfloor \frac{3n}{2}\rfloor -5$ & $2n-5$  & 3 [Prop.~\ref{pro:rotation-lowerbound}]& 4\\
\hline
Empty-Tri.~Rot. &  $\lfloor \frac{3n}{2}\rfloor -5$ & $2n-5$  & 3& 4 [Thm.~\ref{thm:sim_empty_upper_convex}]\\
\hline
Edge Slide &  $\lfloor \frac{3n}{2}\rfloor -5$ & $2n-5$ [Thm.~\ref{thm:slide_upper_convex}] & $\Omega(\log n)$ [Thm.~\ref{thm:sim_slide_lower_convex}]& $O(\log n)$ [Thm.~\ref{thm:sim_slide_upper_convex}]\\
\hline
\end{tabular}
\egroup
\end{center}
\caption{Diameter bounds for $n$ points in convex position.}\label{tbl:bounds_convex}
\end{table}

\paragraph{Exchange} (operation~$\ex$, for short).
For $n\geq 4$, $n$ points in convex position admit (at least) two edge-disjoint spanning trees in $\mathcal{T}(S)$. Since each elementary operation replaces only one edge, this yields a trivial lower bound of $n-1\leq f_{\ex}(n)$ for the diameter of the transition graph. Hernando et al.~\cite{HERNANDO1999} gave a lower bound of $\lfloor 3n/2\rfloor -5\leq f_{\ex}^{\cx}(n)$ for $n$ points in convex position. An upper bound of $f_{\ex}(n)\leq 2n-4$, $n\geq 2$, for points in general position follows from an algorithm by Avis and Fukuda~\cite{avis1996reverse},
in which all exchange operations are in fact rotations (see discussion below).
In the simultaneous setting, the lower and upper bound of $1$ is clear: Given two trees $T_1,T_2\in \mathcal{T}(S)$, one can remove all edges of $E_1\setminus E_2$ and insert all edges of $E_2\setminus E_1$ simultaneously (with an arbitrary bijection between these edge sets). In particular, the simultaneous exchange graph is a complete graph on $|\mathcal{T}(S)|$ vertices for every point set $S$ in general position.

\paragraph{Compatible Exchange} (operation~$\ce$, for short).
For single operations, linear lower and upper bounds for the diameter of the transition graph follow from corresponding bounds for weaker and stronger operations, respectively (cf.~Table~\ref{tbl:bounds_general}).
A simultaneous compatible exchange graph is typically not a complete graph.
Buchin et al.~\cite{buchin2009transforming} constructed a set $S$ of $n$ points and a pair of trees $T_1,T_2\in \mathcal{T}(S)$ such that $\Omega(\log n/\log\log n)$ simultaneous compatible exchanges are required to transform $T_1$ into $T_2$.
Aichholzer et al.~\cite{aichholzer2002sequences} proved that, for every set $S$ of $n$ points, every $T\in \mathcal{T}(S)$ can be transformed into a Euclidean minimum spanning tree of $S$ using $O(\log n)$ simultaneous compatible exchanges; moreover, each operation decreases the Euclidean weight of the tree. Later, Aichholzer et al.~\cite{aichholzer2006transforming} showed that every $T\in \mathcal{T}(S)$ can be transformed into some canonical tree using $O(\log k)$ simultaneous compatible exchanges, where $k\leq \lceil n/3\rceil$ is the number of convex layers of $S$. In particular, $\Omega(\log n/\log\log n)\leq f_{\sce}(n)\leq O(\log n)$, where ${\sce}$ stands for simultaneous compatible exchange. These bounds leave only a sub-logarithmic gap on the asymptotic growth rate of $f_{\sce}(n)$. It is easy to see that  $f_{\sce}^{\cx}(n)=2$. Indeed, a plane spanning tree $T_{1}$ can be transformed into any other plane spanning tree $T_{2}$ by exchanging all edges of $T_{1}$ with the edges of a path $T_0$ along the convex hull, and then exchanging all edges of $T_0$ with $T_{2}$. The existence of two incompatible spanning trees for all $n\geq 4$ implies a lower bound of~$2$.

\paragraph{Rotation} (operation~$\ro$, for short).
The edge rotation operation was first introduced by Chartrand et al.~\cite{CSZ85} for abstract graphs. We consider it over $\mathcal{T}(S)$. For single rotations, the lower bound follows from the corresponding bound for stronger operations. An upper bound $f_{\ro}(n)\leq 2n-4$ follows from a proof by Avis and Fukuda~\cite{avis1996reverse}: They show that every tree in $\mathcal{T}(S)$ can be carried to a star centered at an extreme point of $S$ using at most $n-2$ operations; hence the diameter is bounded by $2(n-2)$. They consider exchange operations, but all exchanges in their proof happen to be rotations.
For simultaneous rotations, $\sro$, we prove the upper bound $f_{\sro}(n)\leq O(\log n)$ (Theorem~\ref{thm:rotations}). A lower bound of $f_{\sro}(n)\geq \Omega(\log n/\log\log n)$ follows from the corresponding bound for the stronger operation of simultaneous compatible exchanges.

For simultaneous rotations and convex position, an algorithm for the weaker operation of empty-triangle rotations yields an upper bound of 4, and we establish a lower bound of 3 (Proposition~\ref{pro:rotation-lowerbound}).

\paragraph{Empty-Triangle Rotation} (operation~${\er}$, for short).
Empty-triangle rotation is a very natural variant of rotation; however, there is not much known about it.
Cano et al.~\cite{cano2013edge} considered empty-triangle rotations over all noncrossing straight-line graphs on a set $S$ of $n$ points with $m$ edges, where $m$ is less than the number of edges in a triangulation of $S$.
They showed that the corresponding transition graph is connected, its diameter is $O(n^2)$, and this bound is the best possible when $m=3n-O(1)$. In the special case $m=n-1$, their result implies that a sequence of $O(n^2)$ empty-triangle rotations can transform a tree in $\mathcal{T}(S)$ into any other tree in $\mathcal{T}(S)$; the intermediate graphs are noncrossing straight-line graphs but they are not necessarily spanning trees.

For single operations, the lower bounds $f_{\er}(n)\geq \lfloor\frac{3n}{2}\rfloor -5$ and $f_{\er}^{\cx}(n)\geq \lfloor\frac{3n}{2}\rfloor -5$ follow from the corresponding bounds for stronger operations. For point sets in general position, we prove an upper bound of $f_{\er}(n)\leq O(n\log n)$ (Theorem~\ref{thm:empty_rotation_general}). For the convex case, we provide a linear upper bound for the weaker operation of edge slide, which yields $f_{\mathrm{er}}^{\cx}(n)=\Theta(n)$.
In the simultaneous setting, we provide a lower bound of $f_{\ser}(n)=\Omega(\log n)$
and an upper bound of $f_{\ser}(n)\leq 8n$ in Theorems~\ref{thm:sim_empty_upper} and~\ref{thm:sim_empty_lower}, respectively. For the case of convex position, we prove $f_{\ser}^{\cx}(n)=\Theta(1)$; see Theorem~\ref{thm:sim_empty_upper_convex}.

\paragraph{Edge Slide} (operation~$\es$, for short).
Aichholzer, Aurenhammer, and Hurtado~\cite{aichholzer2002sequences} proved that $\mathcal{G}_{\es}(S)$, is connected for every point set $S$ in general position.
Aichholzer and Reinhardt~\cite{aichholzer2007quadratic} proved $f_{\es}(n)=\Theta(n^2)$.
For point sets in convex position, we show that $f_{\es}^{\cx}(n)=\Theta(n)$; see Theorem~\ref{thm:slide_upper_convex}.
The simultaneous variant has not been previously considered.
A linear lower bound and a quadratic upper bound can be easily derived from diameter bounds for single operations over point sets in general position, as will be discussed in Section~\ref{sec:edgeslide_general}. For points in convex position, however, we prove an asymptotically tight bound $f_{\ses}^{\cx}(n)=\Theta(\log n)$; see Theorems~\ref{thm:sim_slide_lower_convex} and~\ref{thm:sim_slide_upper_convex}.

\paragraph{Labeled Edges.}
The current best diameter bounds for the transition graphs of edge-labeled noncrossing straight-line spanning trees on $n$ vertices are
summarized in Tables~\ref{tbl:labeled_general} and~\ref{tbl:labeled_convex}. Our contributions are marked with the corresponding theorems and propositions. Some of the bounds immediately follow from the corresponding bounds for the unlabeled variants. Some other bounds are proved using a combination of new ideas and adaptations of results and techniques from prior work, namely~\cite{aichholzer2002sequences,aichholzer2007quadratic,avis1996reverse,BLPV18,cano2013edge,Garcia14,dual_diameter,pgmp-eptvsp-91,STT92}.

\begin{table}[H]
\begin{center}
\bgroup
\def\arraystretch{1.15}
\begin{tabular}{|l|l|l|l|l|}
\hline
Operation& Single Oper. & Single Oper. & Simultaneous & Simultaneous  \\
 &  Lower Bd. & Upper Bd.  & Lower Bd. & Upper Bd.\\
\hline \hline
Exchange &$\lfloor\frac{3n}{2}\rfloor -5$~\cite{HERNANDO1999} & $11n-22$  & 3 [Prop.~\ref{pro:labeled-sim-lower}] &  3 [Prop.~\ref{pro:labeled-sim-exchange} \&~\ref{pro:labeled-cx-sim-exchange}]\\
\hline
Compatible Ex. & $\lfloor\frac{3n}{2}\rfloor -5$ & $11n-22$ & $\Omega(\frac{\log n}{\log \log n})$ \cite{buchin2009transforming} & $O(\log n)$ [Prop.~\ref{pro:labeled-sim-exchange}] \\
\hline
Rotation & $\lfloor\frac{3n}{2}\rfloor -5$ & $11n-22$ [Thm.~\ref{thm:labeled-rotation}] & $\Omega(\frac{\log n}{\log \log n})$ & $O(n)$\\
\hline
Empty-Tri.~Rot. & $\lfloor\frac{3n}{2}\rfloor -5$ & $O(n\log n)$~[Thm.~\ref{thm:labeled_empty_triangle_rotation_upper}] & $\Omega(\log n)$ [Thm.~\ref{thm:sim_empty_lower}] & $O(n)$~[Thm.~\ref{thm:labeled_sim_empty_triangle_rotation_upper}]\\
\hline
Edge Slide & $\Omega(n^{2})$~\cite{aichholzer2007quadratic} & $O(n^{2})$~[Thm.~\ref{thm:labeled_edge_slides_upper}]  & $\Omega(n)$  [Prop.~\ref{pro:sim-slide}]& $O(n^{2})$ [Thm.~\ref{thm:labeled_edge_slides_upper}]\\
\hline
\end{tabular}
\egroup
\end{center}
\caption{Diameter bounds for labeled spanning trees for $n$ points in general position.}\label{tbl:labeled_general}
\end{table}

\begin{table}[H]
\begin{center}
\bgroup
\def\arraystretch{1.15}
\begin{tabular}{|l|l|l|l|l|}
\hline
Operation&  Single Oper. & Single Oper. & Simultaneous & Simultaneous  \\
 & Lower Bd. & Upper Bd. & Lower Bd. & Upper Bd.\\
\hline \hline
Exchange   & $\lfloor\frac{3n}{2}\rfloor -5$~\cite{HERNANDO1999} & $6n-13$ & 3 [Prop.~\ref{pro:labeled-sim-lower}]& 3 [Prop.~\ref{pro:labeled-cx-sim-exchange}]\\
\hline
Compatible Ex.  &  $\lfloor\frac{3n}{2}\rfloor -5$ & $6n-13$ & 3 & 4 [Prop.~\ref{pro:labeled-cx-sim-exchange}]\\
\hline
Rotation &  $\lfloor\frac{3n}{2}\rfloor -5$ & $6n-13$ & $3$ & $O(\log n)$\\
\hline
Empty-Tri.~Rot. &  $\lfloor\frac{3n}{2}\rfloor -5$ & $6n-13$ [Prop.~\ref{pro:labeled-cx-emptytriangle}] & $3$& $O(\log n)$ [Prop.~\ref{pro:labeled-cx-emptytriangle}]\\
\hline
Edge Slide  &  $\Omega(n\log n)$ [Thm.~\ref{thm:labeled-edgeslide-cx-lower}] & $O(n \log n)$ [Thm.~\ref{thm:labeled-edgeslide-cx-upper}]& $\Omega(\log n)$ [Thm.~\ref{thm:labeled-edgeslide-cx-lower}]& $O(n)$ [Thm.~\ref{thm:labeled-edgeslide-cx-upper}] \\
\hline
\end{tabular}
\egroup
\end{center}
\caption{Diameter bounds for labeled spanning trees for $n$ points in convex position.}\label{tbl:labeled_convex}
\end{table}

\subsection{Further Related Work}
\label{ssec:previous}

Exchanging edges such that both the initial and the resulting graph belong to the same graph class is a well-studied operation in various contexts;
see~\cite{flip_survey} for a survey.
Perhaps the best known operation on trees is the classic \emph{rotation} on \emph{ordered rooted binary trees}, which is equivalent to the associativity rule over $n$-symbol words, and to edge flips in triangulations of $n+2$ points in convex position. Sleator, Tarjan, and Thurston~\cite{STT88} gave an upper bound of $2n-10$ for the diameter of the transition graph for $n\geq 13$, later Pournin~\cite{Pournin14} gave a purely combinatorial proof and showed that this bound is tight for all $n\geq 13$.

For abstract trees on $n$ labeled vertices, any spanning tree $T_1=(S,E_1)$ can be transformed into any other tree $T_2=(S,E_2)$ using $|E_1\setminus E_2|$ exchange operations, by the classic exchange property of graphic matroids (see, e.g.,~\cite{Oaxley1993}). Consequently, the diameter of the transition graph is $n-1$.

There are $n^{n-2}$ abstract spanning trees on $n$ labeled vertices for $n\geq 3$~\cite{Cayley89}.
In contrast, the number of noncrossing straight-line trees on $n$ points in the plane is in $O(141.07^n)$~\cite{HSSSTW13} and $\Omega(6.75^n)$~\cite{AHH+07,FN99}.
While the transition graph of the exchange operation over $\mathcal{T}(S)$ is a subgraph of the transition graph over abstract labeled trees (it has fewer nodes), this does not imply any relation between the diameters of these transition graphs.

The operations of exchange, rotation, and edge slide on \emph{unlabeled} abstract spanning trees on $n$ vertices
were considered by Faudree et al.~\cite{FAUDREE1994}, and by Goddard and Swart~\cite{GODDARD1996}.
They define transition graphs over isomorphism classes, proving upper bounds of $n-3$, $2n-6$, and $2n-6$, respectively, on their diameters. For all three operations, a lower bound of $n-3$ is established by the distance between a path and a star.

Geometric variants, where the vertex set $S$ is a set of points in the plane, were first considered by Avis and Fukuda~\cite{avis1996reverse} for the efficient enumeration of all trees in $\mathcal{T}(S)$.
Interestingly, the \emph{$x$-type of $S$} (i.e., the intersection graph of the edges in the straight-line drawing of the complete graph on $S$) can be reconstructed from the transition graph of exchanges~\cite{Keller2016}, or compatible exchange~\cite{OropezaT19}; if $S$ is in convex position, the $x$-type is already determined by the exchanges on spanning paths~\cite{KellerS18}.

Akl et al.~\cite{AKL2007} and Chang and Wu~\cite{CHANG2009} considered the exchange operation over $\mathcal{P}(S)$, the set of noncrossing spanning paths on $n$ points in convex position.
They proved that the diameter of the transition graph is $2n-6$ for $n\geq 5$ and $2n-5$ for $n\in \{3,4\}$.
Wu et al.~\cite{WU2011} used these operations for generating all paths in $\mathcal{P}(S)$ in $O(1)$ amortized time per path.
It remains an open problem whether the exchange graph of $\mathcal{P}(S)$ is connected for general point sets~$S$. Under the weaker operation of edge slides, however, the transition graph of $\mathcal{P}(S)$ is disconnected for $n\geq 4$, since an edge can slide only if it is incident to one of the two leaves.

We are perhaps the first to study the diameter of the transition graph in the edge-labeled variant.
However, transition graphs of edge flips in edge-labeled triangulations of a point set have been studied extensively.
Recently, Bose et al.~\cite{BLPV18} considered the orbits of individual edges.
Lubiw et al.~\cite{LMW19} proved that a sequence of $O(n^7)$ flips can carry any edge-labeled triangulation to any other (by showing that the 2-skeleton of the flip complex is contractible).
Cano et al.~\cite{cano2013edge} considered empty-triangle rotations over edge-labeled noncrossing planar straight-line graphs on a set $S$ of $n$ points with $m$ edges, where $m$ is less than the number of edges in a triangulation of $S$.
They  proved that the transition graph is connected, but did not establish upper and lower bounds on the diameter.

\section{Rotation}
\label{sec:Rotation}

\subsection{General Position}

In this section, we prove the upper bound $f_{\sro}(n)=O(\log n)$ under simultaneous rotations $\sro$ (see Theorem~\ref{thm:rotations}). We bound the diameter of the transition graph by an algorithm that transforms every tree $T\in \mathcal{T}(S)$ into a star, combining ideas from~\cite{aichholzer2002sequences} and~\cite{avis1996reverse}. Our upper bound does not match the lower bound of $\Omega(\log n/\log\log n)$,
which is derived from the stronger simultaneous compatible exchanges.

\begin{theorem}\label{thm:rotations}
Every plane tree in $\mathcal{T}(S)$, $|S|=n$, can be transformed into any other tree in $\mathcal{T}(S)$ using $O(\log n)$ simultaneous rotations;
that is, $f_{\sro}(n)=O(\log n)$.
\end{theorem}
\begin{proof}
Let $S$ be a set of $n$ points in general position, and let $p$ be an extremal point in $S$. We show that every $T = (S,E)\in \mathcal{T}(S)$ can be transformed into a star centered at $p$ using $O(\log n)$ simultaneous rotations, which readily implies $f_{\sro}(n)=O(\log n)$.

We define a simultaneous compatible exchange, $\starify$, on $\mathcal{T}(S)$, and then show that
\begin{enumerate}\itemsep 0pt
\item $O(\log n)$ successive $\starify$ operations can transform every $T\in \mathcal{T}(S)$ into the star centered at $p$, and
\item each $\starify$ operation can be replaced by at most four simultaneous rotations.
\end{enumerate}

\paragraph{Preliminaries.}
For convenience, we embed the Euclidean plane $\mathbb{R}^2$ into the real projective plane $P\mathbb{R}^2$
by adding a line ``at infinity.'' Apply a projective transformation that maps $p$ to the point at $y=-\infty$ at infinity, and a tangent line of $\conv(S)$ incident to $p$ to the line at infinity. Note that a line segment $pq$ becomes a vertical downward ray emanating from $q$. Since $S$ is in general position, no two points in $S\setminus \{p\}$ have the same $x$-coordinate. Denote by $x_q$ the $x$-coordinate of a point $q\in S\setminus \{p\}$.
For $q,r\in S\setminus \{p\}$, let
\begin{equation*}
\width(qr)=|\{s\in S\setminus \{p\}: \min(x_q,x_r)\leq x_s<\max(x_q,x_r)\}|,
\end{equation*}
that is, the number of points in $S\setminus \{p\}$ lying in the left-closed right-open vertical slab spanned by $qr$,
that we denote by $W(qr)$. If $r$ or $q$ equals the point $p$ at infinity, we define $\width(qr)=0$. Clearly, $\width(qr)\leq |S|-2$ for all $q,r\in S$ (as neither $p$ nor the rightmost point in $S\setminus \{p\}$ is contained in any of these slabs). Note also that for any sequence of points $q_0,\ldots , q_t\in \mathbb{R}^2$ sorted by $x$-coordinates, we have $\width(q_0q_t)=\sum_{i=1}^t \width(q_{i-1}q_i)$.

The edges in $E$ can be ordered as follows. First define a binary relation~$\prec$ on $E$
such that $e_1\prec e_2$ if $W(e_1)\cap W(e_2)\neq \emptyset$  (i.e., the slabs $W(e_1)$ and $W(e_2)$ overlap),
and $e_1$ is below $e_2$ within $W(e_1)\cap W(e_2)$.
This relation is acyclic (every nonempty set of edges has a minimum element,
as the lower envelope of noncrossing segments contains one of the segments).
Consequently, its transitive closure is a poset.
Fix an arbitrary linear extension of this partial order.

\paragraph{Definition of operation $\starify$.}
Let $T=(S,E)\in \mathcal{T}(S)$. Refer to \figurename~\ref{fig:vertical}.
Designate $p$ as the root of $T$, and direct its edges toward the root.
For every vertex $s\in S\setminus \{p\}$, denote by $e_s\in E$ the unique outgoing edge.
We will rotate each edge $e_s$ to some edge $e'_s$ incident to $s$ (possibly, $e_s=e'_s$) such that the union of all old and new edges $\bigcup_{s\in S\setminus \{p\}} \{e_s,e'_s\}$ forms a noncrossing graph (hence $\starify$ is a simultaneous compatible exchange operation).

From every $s\in S\setminus \{p\}$, draw a vertical downward ray~$r_s$ until it either reaches $p$ (at infinity) or crosses some edge in $E$.
If $r_s$ reaches $p$, then let $e'_s=sp$ (possibly $e'_s=e_s)$. It remains to define the image $e'_s$ for all other edges $e_s\in E$.
For every edge $e\in E$, let $S_e$ be the set of vertices $s\in S\setminus \{p\}$ such that $r_s$ hits the interior of $e$. For every edge $e\in E$ where $S_e\neq \emptyset$, we create an $x$-monotone polygon $P_e$ bounded by two $x$-monotone chains: the lower chain consists of the single edge $e$, and the upper chain connects the endpoints of $e$ via the points in $S_e$ sorted by increasing $x$-coordinates; see~\figurename~\ref{fig:vertical}.
In particular, the upper chain consists of precisely $|S_e|+1$ line segments.
For the operation $\starify$, remove an edge of the upper chain that has maximum width, direct the resulting two $x$-monotone paths to the two endpoints of $e$, and define $e'_s$ for all $s\in S_e$ to be the unique outgoing edge along these paths. This completes the description of operation $\starify$.

\begin{figure}[htbp]
\centering
\includegraphics{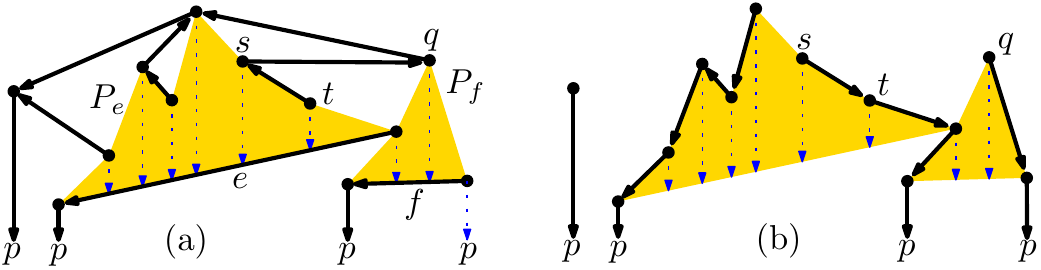}
\caption{(a) A straight-line spanning tree $T$ where $p=(0,-\infty)$, dotted vertical downward rays,
and the shaded polygons $P_e$ and $P_f$ for edges $e$ and $f$. (b) The result of operation $\starify$.
This operation is not a simultaneous rotation: Edge $st$ is an edge of both $T$ and $\starify(T)$,
so it cannot be rotated to any other edges within a simultaneous rotation.}
\label{fig:vertical}
\end{figure}

\paragraph{Correctness.}
We first show that operation $\starify$ is a simultaneous compatible exchange operation. (Later, we show how to model a $\starify$ operation with up to four simultaneous rotations.) We need to show that if $T\in \mathcal{T}(S)$, then $\starify(T) \in \mathcal{T}(S)$, and the edges of $T$ and $\starify(T)$ do not cross.
We start by proving the following claim.

\begin{clm}\label{clm:noncrossing}
For every edge $e\in E$, where $S_e\neq \emptyset$, the interior of the polygon $P_e$ is disjoint from the edges in $E$.
\end{clm}
\begin{proof}
Suppose, to the contrary, that an edge in $E$ intersects the interior of $P_e$.
If there is an edge $e'\in E$ that crosses the boundary of $P_e$ (at least) twice, then both crossings are on the upper chain of $P_e$, since the edges in $E$ are noncrossing, so $e'$ cannot cross $e$.
But then the upper chain of $P_e$ has at least one vertex $u$ between the two crossings with $e'$, and the vertical downward ray from $u$ would hit $e'$ before $e$, contradicting our assumption that $u\in S_e$.
Suppose now that some edge in $E$ crosses the boundary of $P_e$ once. Let $e''\in E$ be the minimal such edge in the linear extension of $\prec$, and let $v''$ be the (unique) endpoint of $e''$ in the interior of $P_e$.
By the minimality of $e''$, the vertical downward ray from $v''$ hits the edge $e$;
hence $v''\in S_e$, which contradicts the assumption that $v''$ lies in the interior of $P_e$.
This completes the proof of Claim~\ref{clm:noncrossing}.
\end{proof}

By Claim~\ref{clm:noncrossing}, the edges in $E$ do not cross any edges of the polygons $P_e$, where $S_e\neq \emptyset$. Since the polygons $P_e$, $e\in E$, are pairwise interior-disjoint, the edges of $\starify(T)$ do not cross each other.

It remains to show that $\starify(T)$ is a spanning tree. By construction, the number of edges remains the same, and every vertex in $S\setminus \{p\}$ has an outgoing edge. So it is enough to show that the graph $\starify(T)$ contains a directed path from every vertex in $S\setminus \{p\}$ to $p$. Recall that we have ordered the edges in $E$ consistently with the above-below relationship. For each edge~$e \in E$, the vertices in $S_e$ are connected to the endpoints of $e$ in $\starify(T)$. Even though an edge $e\in E$ may not be present in $\starify(T)$, the two  endpoints of $e$ each have an outgoing edge to some endpoint of some edges $e_1\in E$ and $e_2\in E$, resp.,
with $e_1\prec e$ and $e_2\prec e$ (possibly, $e_1 = e_2$), or directly to~$p$.
Consequently, $\starify(T)$ contains a directed path from every vertex in $S\setminus \{p\}$ to $p$.

\paragraph{A sequence of $O(\log n)$ $\starify$ operations.}
Let $T_0\in \mathcal{T}(S)$. For $i\geq 1$, let $T_i=\starify(T_{i-1})$ and let $E_i$ be the edge set of $T_i$.
We need to show that $T_k$ is a star centered at $p$ for all $k\geq \lceil \log_2 n\rceil$.
To this end, we prove the following claim.

\begin{clm}\label{clm:log}
If $e'_s\in E_{i+1}$ is an outgoing edge of $s\in S$ and $e'_s$ is not incident to $p$,
then $s\in S_e$ for some edge $e\in E_i$ such that $\width(e)\geq 2\cdot\width(e'_s)$.
\end{clm}
\begin{proof}
By construction, if $e'_s$ is not incident to $p$, then it is an edge on the upper chain of some polygon $P_e$, $e\in E_i$, where $s\in S_e$.
By construction, $\width(e)$ equals the sum of widths of the edges in the upper chain of polygon $P_e$.
Since we do not use an edge of maximum width in this chain, we have
$\width(e'_s)\leq \frac{1}{2}\width(e)$ for every vertex $s\in S_e$, as claimed.
\end{proof}

Now suppose that $k\geq \lceil \log_2 n\rceil$ and $\width(e_1)\geq 1$ for some $e_1\in E_k$.
By Claim~\ref{clm:log}, there is a chain of edges $e_i\in E_{k+1-i}$ for $i=1,2,\ldots, k$
such that $\width(e_{i+1})\geq 2\cdot\width(e_i)$.
This implies that $\width(e_0)\geq 2^k$, where $2^k>n$,
which contradicts the fact that $\width(e)\leq n-2$ for any edge~$e$.
This proves that $\width(e)=0$ for all $e\in E_k$ if $k\geq \lceil \log_2 n\rceil$.
Consequently, $T_k$ is a star centered at $p$, as claimed.

\paragraph{Implementation of $\starify$ with four simultaneous rotations.}
We have seen that $\starify$ is a simultaneous compatible exchange operation.
However, it need not be a simultaneous rotation.
Indeed, consider an edge $e\in E$ where $S_e\neq\emptyset$.
Operation $\starify$ transforms every edge $e_s$, $s\in S_e$, into some edge $e'_s$ on the upper chain of the polygon $P_e$. This operation is not necessarily a simultaneous rotation: For example, if $e_s\neq e'_s$ and $e'_s$ is already present in $T$ (cf.~\figurename~\ref{fig:vertical}), then a simultaneous rotation cannot transform $e_s$ into $e'_s$ directly (recall that a simultaneous operation between $T_1$ and $T_2$ requires a bijection between $E_1\setminus E_2$ and $E_2\setminus E_1$). We now show that $\starify$ can be implemented by a sequence of up to four simultaneous rotations.

We define the four simultaneous rotations for the outgoing edges of each point set $S_e$ independently (in the total order on the edges in $E$ defined above). Suppose $e=(u,v)$, where $e=e_u$ (note that edge $e_v$ may lie on the boundary of $P_e$). Triangulate $P_e$ arbitrarily. We call a vertex in $S_e$ a \emph{peak} if it is incident to only one triangle, and \emph{nonpeak} otherwise. By definition two peak vertices in $S_e$ cannot be consecutive vertices of $P_e$. The dual graph of the triangulation of the polygon $P_e$ is a tree. By a BFS traversal of the dual graph (tree) starting from the triangle adjacent to $e$ (root), each triangle is adjacent to either edge $e$ or an edge of the previous (i.e., parent) triangle. Assign each triangle $\Delta_s$ to its vertex $s\in S_e$ that is an endpoint of neither $e$ nor a previous triangle. For every vertex $s\in S_e$, let $A_s$ denote the set of two edges of $\Delta_s$ incident to $s$, and by $b_s$ the edge of $\Delta_s$ opposite to $s$. Note that, for every peak vertex $s\in S_e$, both edges in $A_s$ lie on the boundary of $P_e$, and for every nonpeak vertex $s\in S_e$, at least one edge in $A_s$ is a diagonal of $P_e$.

\begin{figure}[htbp]
\centering
\includegraphics{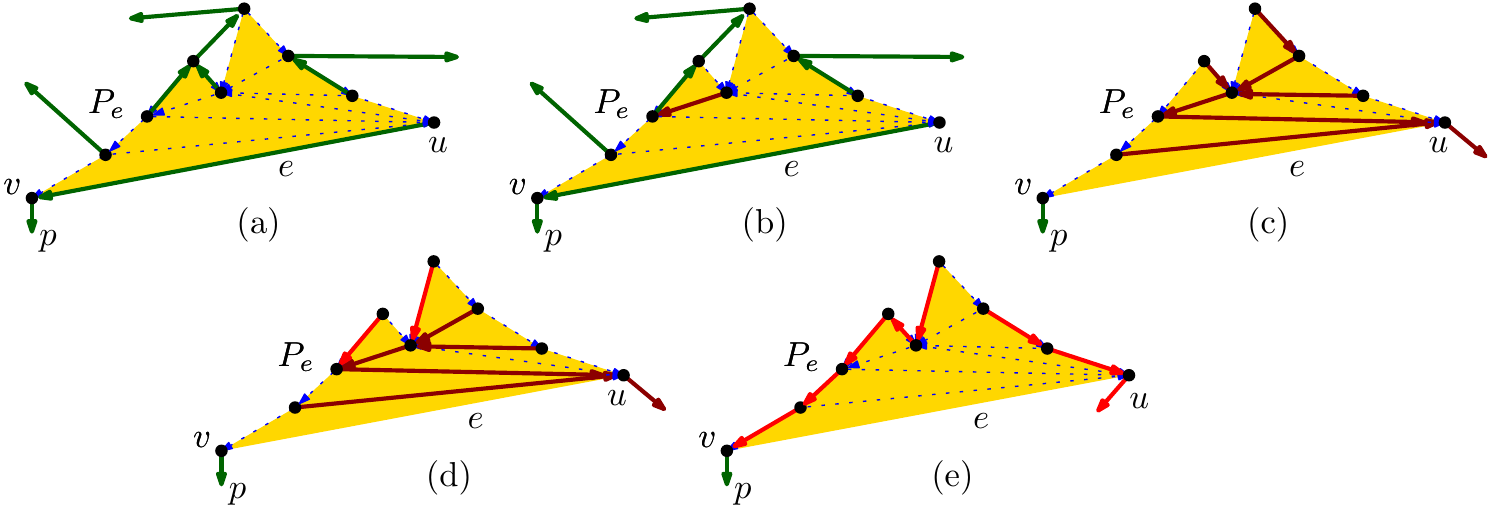}
\caption{(a) A triangulation of polygon $P_e$ and the edges in $E$ incident to the vertices of $P_e$.
(b--e) The result of the 1st, 2nd, 3rd, and 4th simultaneous rotation, respectively.
Arrows pointing away from $P_e$ end at vertices that are not shown in these figures.
}
\label{fig:simulation}
\end{figure}

We can now describe four simultaneous rotations for the vertices in $S_e$, $e\in E$ (see~\figurename~\ref{fig:simulation}).

\begin{enumerate}\itemsep 0pt
\item\label{step:1} For every vertex $s\in S$,
    \begin{enumerate}
    \item\label{step:1a} if the downward vertical ray from $s$ does not hit any edge in $T$, then rotate $e_s$ to $sp$, and
    \item\label{step:1b} if $s$ is a peak vertex in $P_e$, for some $e\in E$, such that both edges in $A_s$ are present in $T$ and directed to $s$, then let $s'$ be the (nonpeak) neighbor of $s$ along the boundary of $P_e$ such that $b_s\in A_{s'}$, and rotate $e_{s'}$ to $b_s$.
    \end{enumerate}
\item\label{step:2} For every vertex $s\in S_e$, $e\in E$, rotate the outgoing edge to an edge in $A_s$ that is (i) a diagonal of $P_e$ if $s$ is nonpeak, and (ii) not in the current tree if $s$ is a peak (break ties arbitrarily).
\item\label{step:3} For every peak vertex $s\in S_e$, $e\in E$, rotate $e_s$ to the edge prescribed by $\starify$.
\item\label{step:4} For every nonpeak vertex $s\in S_e$, rotate $e_s$ to an edge of $P_e$ prescribed by $\starify$.
\end{enumerate}

For the correctness of the four simultaneous rotations, we need to show that each operation produces a noncrossing directed spanning tree rooted at $p$. First we show that none of the operations rotates an edge to another edge of the current tree. Step~\eqref{step:1a} rotates edges to their positions prescribed by $\starify$. Steps~\eqref{step:1b}--\eqref{step:3} rotate edges $e_s$, $s\in S_e$, to an edge in $A_s$.
Note that the sets $A_s$ are pairwise disjoint, and each edge in $A_s$ lies on the boundary or in the interior of $P_e$.
By Claim~\ref{clm:noncrossing}, the edges in $E$ do not intersect the interior of any polygon $P_e$; the same holds for all edges created in Step~\eqref{step:1a}.
Consequently, Steps~\eqref{step:1b}-\eqref{step:2} do not rotate any edge to another edge.
At the end of Step~\eqref{step:3}, the outgoing edges of peak vertices in $S_e$, $e\in E$, are already at their final positions prescribed by $\starify$; the outgoing edges of nonpeak vertices are diagonals of $P_e$.
Therefore, Step~\eqref{step:4} does not rotate any edge to an existing edge, either.

Step~\eqref{step:1} clearly maintains a directed spanning tree rooted at $p$. At the end of Step~\eqref{step:2}, the outgoing edges of the vertices in $S_e$, $e\in E$, induce two forests rooted at the endpoints of $e$; and both Step~\eqref{step:3} and \eqref{step:4} maintain this property. Since each endpoint of $e$ is either adjacent to $p$ or is part of some set $S_{e'}$ with $e'\prec e$, this property implies a directed path from every vertex in $S\setminus \{p\}$ to $p$.
\end{proof}

\subsection{Convex Position}

For simultaneous rotations and point sets in convex position, the upper bound of 4 results from an algorithm for the weaker operation of simultaneous empty-triangle rotations. Below, we establish a lower bound of 3 for simultaneous rotations and $n\geq 6$ points in convex position.

\begin{propo}\label{pro:rotation-lowerbound}
For every set $S$ of $n\geq 6$ points in convex position, there exist two trees in $\mathcal{T}(S)$
such that it takes at least 3 simultaneous rotations to transform one into the other;
that is, $f_{\sro}^{\cx}(n)\geq 3$.
\end{propo}
\begin{proof}
For $n=6$, let $T_1$ and $T_2$ be the two trees shown in \figurename~\ref{fig:rotation-lowerbound}. Consider a sequence of simultaneous rotations that transform $T_1$ into $T_2$. The first simultaneous rotation either keeps $ad$ in place or rotates it to $ac$, $ae$, $bd$, or $fd$. In all cases, edge $ad$ or its image crosses some edge of $T_2$.
Consequently, at least two more simultaneous operations are needed to reach $T_2$.
For $n>6$, we may augment $T_1$ and $T_2$ with $n-6$ vertices between $a$ and $b$, and the same argument shows that the distance between the two trees is at least 3.
\end{proof}

\begin{figure}[htbp]
\centering
\includegraphics{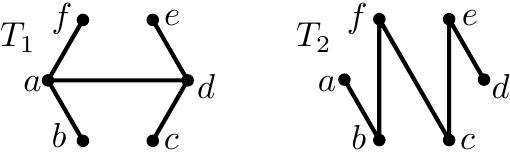}
\caption{Two spanning trees on 6 points in convex position.
}
\label{fig:rotation-lowerbound}
\end{figure}

\section{Empty-Triangle Rotation}
\label{sec:EmptyTriangleRotation}

\subsection{General Position}
\label{general and er}

For single operations, the lower bound of $\lfloor \frac{3n}{2}\rfloor-5$ follows from an analogous bound for the stronger operation of exchange. We prove an upper bound of $O(n \log n)$ (see Theorem~\ref{thm:empty_rotation_general}), which leaves a logarithmic gap. We start with an easy observation about a single triangle.

\begin{propo}\label{pro:replacing}
Let $T=(S,E)$ be a spanning tree with three vertices $p,q,r\in S$ such that $pq\in E$ and the interior of the triangle $\Delta(pqr)$ does not intersect any edge of $T$. Then an empty-triangle rotation can replace $pq$ with either $pr$ or $qr$.
\end{propo}
\begin{proof}
Since $T$ is a tree, $pq$ is a bridge. The graph $T-pq$ is a forest of two trees in which $p$ and $q$ are in distinct components. We can exchange $pq$ for the edge $rq$ if $r$ is in the same component as $p$, otherwise we can exchange $pq$ for the edge $pr$, to obtain a noncrossing spanning tree. In either case, the exchange is an empty-triangle rotation.
\end{proof}

\begin{theorem}\label{thm:empty_rotation_general}
Every plane tree in $\mathcal{T}(S)$, $|S|=n$, can be transformed into any other tree in $\mathcal{T}(S)$ using $O(n\log n)$ empty-triangle rotations;
that is, $f_{\er}(n)=O(n\log n)$.
\end{theorem}
\begin{proof}
Let $T$ be a spanning tree in $\mathcal{T}(S)$ for a point set $S$ of size $n$ and let $p \in S$ be an extremal point in $S$.
We show that we can transform $T$ into a star centered at $p$ using $O(n \log n)$ empty-triangle rotations.
To this end, we use $O(n)$ operations to transform $T$ into two subtrees of roughly equal size whose convex
hulls intersect in $p$ only, and then recurse on the subtrees.

Let $h$ be a ray emanating from $p$ that subdivides the convex hull of $S$ into two parts,
neither containing more than $n/2$ points of $S\setminus \{p\}$.
If $h$ does not cross any edge of $T$, we can recurse on the two subtrees.
Otherwise, let $e$ be the edge of $T$ whose crossing with $h$ is farthest away from~$p$.
Triangulate $T$ (i.e., augment $T$ into an edge-maximal planar straight-line graph).
By Euler's polyhedron formula, the triangulation has at most $2n-5$ bounded faces.
Let $(\Delta_1, \dots, \Delta_m)$ be the sequence of bounded faces (triangles) of the triangulation
that intersect the line segment $(p,h\cap e)$, in the order in which they are visited by $h$.
Note that $m \leq 2n-5$. By Proposition~\ref{pro:replacing}, an empty-triangle rotation can
replace $e$ with some other edge $f$ of $\Delta_m$. See \figurename~\ref{fig:Recursive_ER}. This edge $f$
either does not cross $h$, or its crossing $h\cap f$ is closer to $p$ than $h\cap e$ is.
In both cases, we obtain a tree $T'\in \mathcal{T}(S)$ whose edge set is contained in the same triangulation;
however, the sequence of triangles visited by~$h$ until the last crossing with an edge in $T'$ is now $(\Delta_1, \dots, \Delta_{m'})$ for some $m' < m$. Consequently, after at most $m \leq 2n-5$ iterations, $h$ does not cross any edge of the tree,
and we can recurse on the two subtrees, each on at most $n/2 + 1$ vertices.

\begin{figure}[htbp]
\centering
\includegraphics{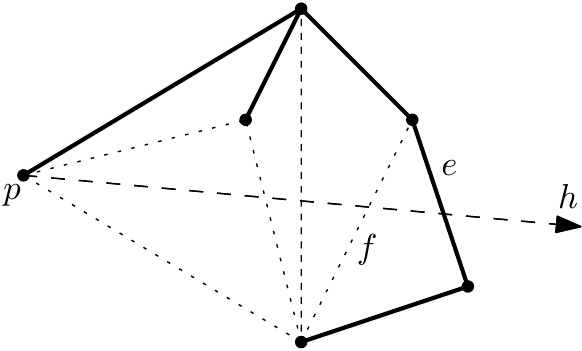}
\caption{An empty-triangle rotation can replace $e$ with $f$.}
\label{fig:Recursive_ER}
\end{figure}

When every subtree contains only two vertices, then all edges are incident to $p$, and their union is a star centered at $p$. The number $a(n)$ of operations needed to transform $T$ into a star centered at $p$ satisfies the recurrence relation $a(n) \leq 2a(n/2 +1)+O(n)$, which solves to $O(n \log n)$. Since any two trees in $\mathcal{T}(S)$ can be transformed into a star centered at $p$ using $a(n)$ operations, we have
$f_{\er}(n)\leq 2a(n)=O(n\log n)$.
\end{proof}

A simultaneous empty-triangle rotation between trees $T_1,T_2\in \mathcal{T}$ requires a bijection between the edges in $E_1\setminus E_2$ and $E_2\setminus E_1$ such that the corresponding edges are adjacent and form empty triangles. The empty triangles involved in such an operation are interior-disjoint. We prove an upper bound of $f_{\ser}(n)\leq O(n)$ in Theorem~\ref{thm:sim_empty_upper}, building on the proof of Theorem~\ref{thm:empty_rotation_general}, and a lower bound of $f_{\ser}(n)\geq O(\log n)$ in Theorem~\ref{thm:sim_empty_lower}.

\begin{theorem}\label{thm:sim_empty_upper}
Every plane tree in $\mathcal{T}(S)$, $|S|=n$, can be transformed into any other tree in $\mathcal{T}(S)$ using less than $8n$ simultaneous empty-triangle rotations; that is, $f_{\ser}(n) < 8n$.
\end{theorem}
\begin{proof}
Revisit the proof of Theorem~\ref{thm:empty_rotation_general}. It takes at most $2n-5$ empty-triangle rotations to split the initial tree into two subtrees, each of size at most $n/2 + 1$.
However, the set of triangles involved in the recursive calls are interior-disjoint, and involve distinct edges, so these rotations can be performed simultaneously. The number of rotations to obtain a star is therefore bounded by the recursion $a(n) \leq 2n-5+a(n/2 +1)$,
which solves to $a(n) < 4n$. Consequently, $f_{\ser}(n)\leq 2a(n)<8n$.
\end{proof}

\begin{theorem}\label{thm:sim_empty_lower}
For every $n\geq 2$, there exist a set $S$ of $n$ points in general position and two trees in $\mathcal{T}(S)$
such that $\Omega(\log n)$ simultaneous empty-triangle rotations are required to transform one into the other;
that is, $f_{\ser}(n)=\Omega(\log n)$.
\end{theorem}
\begin{proof}
We may assume that $n=2^k+1$ for some $k\in \mathbb{N}$. We construct a point set $S$ and two spanning trees $T,T'\in \mathcal{T}(S)$ such that it takes at least $k=\log_2(n-1)$ simultaneous empty-triangle rotations to transform $T$ into $T'$. The points in $S$ have integer coordinates, and are not in general position, but a random perturbation by a small $\varepsilon>0$ would bring $S$ to general position and preserve all combinatorial properties in our proof.

Our point set is $S=\{(x,\varphi(x)): x=0,\ldots, n\}$, where we define $\varphi(x):\{0,\ldots ,n\}\rightarrow \mathbb{N}_0$ as follows (see \figurename~\ref{fig:binary}). Every integer $x\in \{0,\ldots , 2^k=n-1\}$ has a binary representation $x=\sum_{i=0}^k x_i2^i$ with $x_i\in \{0,1\}$. For $x=1,\ldots, n-2$, let $j(x)$ be the smallest index such that $x_{j(x)}=1$, and let $\varphi(x)=n^{2(k-j(x))}$;
and for $x\in \{0,n-1\}$, let $\varphi(0)=\varphi(n-1)=1$. This completes the definition of $S$.
Let $T$ and $T'$, respectively, be stars centered at $p=(0,1)$ and $r=(1,\varphi(1))=(1,n^{2k})$.

\begin{figure}[htbp]
\centering
\includegraphics[width=0.9\textwidth]{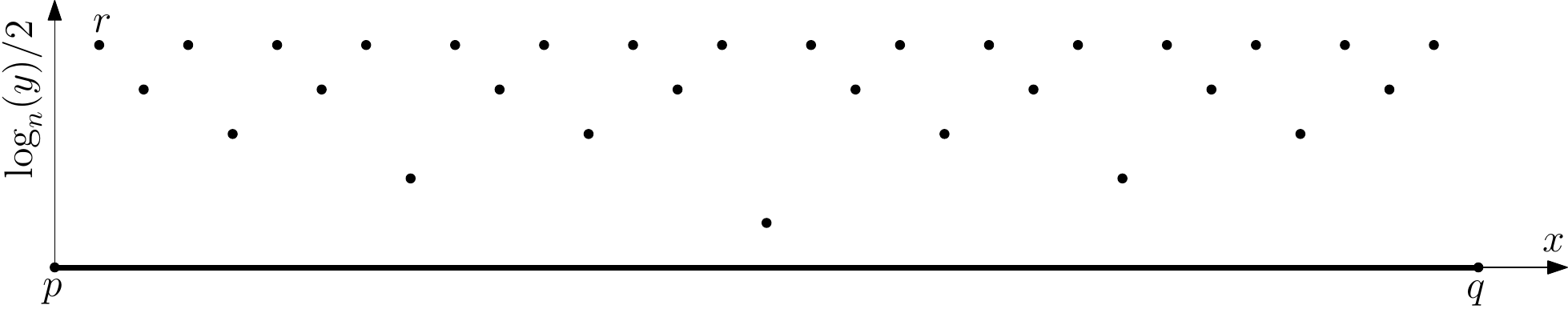}
\caption{A schematic image of the point set $S$ for $n=2^5+1=33$ and $k=5$, on a logarithmic scale.
The edge $pq$ is horizontal and vertex $r=(1,n^{2k})=(1,33^{10})$ has maximal $y$-coordinate.
(The logarithmic scale distorts the slopes.) }
\label{fig:binary}
\end{figure}

Let $(T=T_0,T_1,\ldots , T_m=T')$ be a sequence of trees in $\mathcal{T}(S)$ such that any two consecutive trees are related by a simultaneous empty-triangle rotation. For every simultaneous operation, there exists a bijection between the old and new edges such that corresponding edges are related by an empty-triangle rotation. For all $i=0,\ldots, m-1$, fix such a bijection for the simultaneous empty-triangle rotation between $T_i$ and $T_{i+1}$, and extend it to a bijection between all edges of $T_i$ and $T_{i+1}$ with identity relation on the edges that are present in both trees. Then every edge in $T$ corresponds to a unique edge in $T_i$ ($i=0,\ldots , m$): It corresponds to an edge incident to $p$ in $T_0$ and to an edge incident to $r$ in $T_m$. In the remainder of the proof, we trace the edges  corresponding to $e=pq$, where $q=(n,1)$, and show that it takes at least $k$ empty-triangle rotations to carry $e$ into an edge incident to $r$, consequently $m\geq k$.

For $i=0,1,\ldots ,k-1$, denote by $S_i$ the set of points in $S$ whose $y$-coordinate is at most $n^{2i}$, that is, $S_i=\{(a,\varphi(a))\in S: \varphi(a)\leq n^{2i}\}$.
We make use of the following claim.
\begin{clm}\label{clm:3rd}
If an empty triangle spanned by $S$ has two vertices in $S_i$ ($i=0,\ldots ,k-1$), then the third vertex must be in $S_{i+1}$.
\end{clm}
To prove this claim, we make a few observations about points in $S$ and the slopes of line segments spanned by $S$.
By construction, for any two points $a,b\in S_i$ ($i=0,\ldots ,k-1$), there is a point $d\in S_{i+1}\setminus S_i$, whose $x$-coordinate is between that of $a$ and $b$.
The slope of a segment between points $a=(x_a,y_a)$ and $b=(x_b,y_b)$ is defined as $\slope(ab)=(y_a-y_b)/(x_a-x_b)$.
In particular, for any two points $a,b\in S_i$, we have $|\slope(ab)|\leq n^{2i}$.
For $a\in S_i$ and $b\in S_{i+1}\setminus S_i$, we have
\begin{equation}\label{eq:slope1}
\frac{n^{2i+1}}{2}<\frac{n^{2i+2}-n^{2i}}{n}<|\slope(ab)|< n^{2i+2}.
\end{equation}
For $a\in S_i$ and $b\in S\setminus S_{i+1}$, we have
\begin{equation}\label{eq:slope2}
\frac{n^{2i+3}}{2}<\frac{n^{2i+4}-n^{2i}}{n}<|\slope(ab)|.
\end{equation}

We are now ready to prove Claim~\ref{clm:3rd}.
\begin{proof}[Proof of Claim~\ref{clm:3rd}.]
Consider an empty triangle $\Delta(abc)$ with $a,b\in S_i$, $x_a<x_b$, and $c\in S\setminus S_{i+1}$.
Then there exists a point $d\in S_{i+1}\setminus S_i$ such that $x_a<x_d<x_b$.
Since $d$ lies above both $a$ and $b$, we have
$\slope(bd)<\slope(ab)<\slope(ad)$.
Inequalities~\eqref{eq:slope1} and \eqref{eq:slope2} yield
\begin{equation}\label{eq:slope3}
|\slope(ab)|<|\slope(ad)|<|\slope(ac)|
\hspace{.1in} \mbox{\rm and} \hspace{.1in}
|\slope(ab)|<|\slope(bd)|<|\slope(bc)|
\end{equation}
If $x_a<x_c<x_b$, then \eqref{eq:slope3} readily implies that
$\slope(ab)<\slope(ad)<\slope(ac)$ and $\slope(bc)<\slope(bd)<\slope(ab)$;
consequently $d$ lies in the interior of $\Delta(abc)$.
If $x_b<x_c$ (resp., $x_c<x_a$), then \eqref{eq:slope3} implies
$\slope(ab)<\slope(ad)<\slope(ac)$
(resp., $\slope(bc)<\slope(bd)<\slope(ab)$),
and so $d$ lies again in the interior of $\Delta(abc)$.
In all cases, triangle $\Delta(abc)$ is nonempty, contrarily to our assumption.
This completes the proof of Claim~\ref{clm:3rd}.
\end{proof}

It follows from Claim~\ref{clm:3rd} that a simultaneous empty-triangle rotation transforms every edge spanned by $S_i$ into an edge spanned by $S_{i+1}$, for $i=0,\ldots , k-1$.
In particular, the edge $e=pq$ is spanned by $S_0$, and the point $r$ is in $S_k\setminus S_{k-1}$.
Consequently, it takes at least $k$
empty-triangle rotations to transform edge $e$ into an edge incident to $r$, as claimed.
\end{proof}

\subsection{Convex Position}
\label{convex and er}

A construction in~\cite{HERNANDO1999} designed for the stronger exchange operation yields the lower bound $\lfloor \frac{3n}{2}\rfloor -5$ for single empty-triangle rotations for point sets in convex position.
Similarly, we can derive an upper bound of $f_{\er}^{\cx}(n)\leq2n-5$ from our algorithm for edge slides (Theorem~\ref{thm:slide_upper_convex}).
In Theorem \ref{thm:sim_empty_upper_convex} below, we provide a constant upper bound for simultaneous empty-triangle rotations.

\begin{figure}[htbp]
\centering
\includegraphics{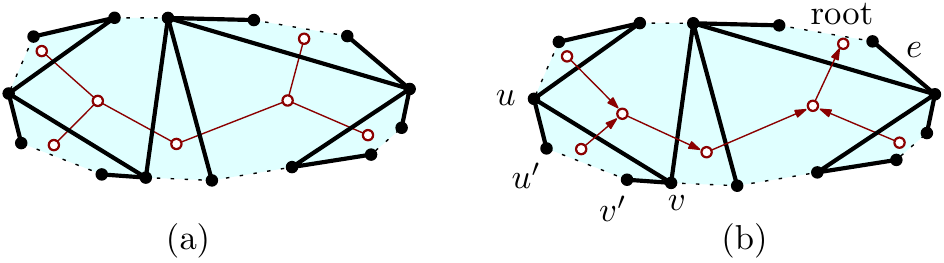}
\caption{(a) A plane tree on a set $S$ of 14 points in convex position, and the dual tree on 7 cells.
(b) An edge $e$ of $\conv(S)$ lies on the boundary of a unique cell, which is the root of the dual tree.}
\label{fig:dualtree}
\end{figure}

We define the \emph{dual tree} of a plane tree $T\in \mathcal{T}(S)$ for a set $S$ of $n\geq 3$ points in convex position as follows; see~\figurename~\ref{fig:dualtree}(a).
The edges of $T$ subdivide the convex $n$-gon $\conv(S)$ into one or more convex \emph{cells}, which correspond to the nodes of the dual tree.
Two nodes of the dual tree are adjacent if the corresponding cells share an edge.
Note that the dual tree is indeed a tree (every edge corresponds to a chord of $\conv(S)$, and so is a bridge).
Furthermore, the boundary of each cell contains precisely one edge that is not in $T$, and this edge is necessarily an edge of $\conv(S)$; we call this edge the \emph{hull edge} of the cell.
The main idea of the proof of the following theorem is to rotate edges shared by cells to hull edges.

\begin{theorem}\label{thm:sim_empty_upper_convex}
For every set $S$ of $n$ points in convex position, every plane tree in $\mathcal{T}(S)$ can be transformed into any other tree in $\mathcal{T}(S)$ using at most $4$ simultaneous empty-triangle rotations; that is, $f_{\ser}^{\cx}(n)\leq 4$.
\end{theorem}
\begin{proof}
Let $S$ be a set of $n\geq 3$ points in convex position, and let $P\in \mathcal{T}(S)$ be a path of $n-1$ arbitrary edges of $\conv(S)$.
We claim that every tree $T\in \mathcal{T}(S)$ can be transformed into $P$ using at most 2 simultaneous empty-triangle rotations. This immediately implies $\diam(G_{\ser}^{\cx}(S))\leq 4$.

To prove the claim, let $T\in \mathcal{T}(S)$ be an arbitrary tree.
Denote by $e$ the edge of $\conv(S)$ that is not in $P$.
The edge $e$ is on the boundary of a unique cell determined by $T$; let this cell be the root of the dual tree and direct all the edges of the dual tree toward the root (see \figurename~\ref{fig:dualtree}(b) for an example).
The boundary of any other cell $C$ contains a unique edge $uv$ that separates it from its parent cell, and it has a unique hull edge $u'v'$ not in $T$.
Since $u$, $v$, $u'$, and $v'$ lie on the boundary of the cell, which is convex, then $\conv(\{u,v,u',v'\})$ is empty.
If $uv$ and $u'v'$ do not share any vertex, then we can use two consecutive empty-triangle
rotations to move $uv$ to $uv'$, and then $uv'$ to $u'v'$.
If $uv$ and $u'v'$ share a vertex,
then a single empty-triangle rotation can move $uv$ to $u'v'$.
Also, if edge $e$ is present in $T$, we can move it to the hull edge of the root cell using at most two operations.
Rotations involving different cells of $T$ can be performed simultaneously.
Consequently, we can transform $T$ into $P$ with two operations, as claimed.
\end{proof}

\section{Edge Slide}
\label{sec:EdgeSlide}

We remark here that for simultaneous edge slides one can consider the following, more restricted variant of simultaneous edge slides. Let $T\in \mathcal{T}(S)$ for a point set $S$ in general position. Two edge slide operations that move $v_1 u_1$ to $v_1 w_1$ and $v_2 u_2$ to $v_2 w_2$, respectively, can be performed simultaneously if the triangles $\Delta(u_1 v_1 w_1)$ and $\Delta(u_2 v_2 w_2)$ intersect in at most one point.
All lower bounds in this section hold for the less restrictive setting (in which $\Delta(u_1 v_1 w_1)$ and $\Delta(u_2 v_2 w_2)$ may share the edge $u_1w_1=u_2w_2$), and the upper bounds apply to the more restricted setting that does not allow a shared edge.

\subsection{General Position}
\label{sec:edgeslide_general}

As noted above, Aichholzer and Reinhardt \cite{aichholzer2007quadratic} proved that $f_{\es}(n)=\Theta(n^2)$.
Little is known about the simultaneous variant.
However, their results immediately imply $f_{\ses}(n)=O(n^2)$.
A lower bound of $f_{\ses}(n)=\Omega(n)$ can also be derived from $f_{\es}(n)=\Omega(n^2)$ if we notice
that a simultaneous edge slide operation can be simulated by a sequence of single edge slide operations
(this property does not hold for the other four operations in this paper).

\begin{lemma}\label{lem:simulate}
Let $T_1,T_2\in \mathcal{S}$ be two trees related by a simultaneous edge slide operation with a bijection $\pi:E_1\setminus E_2\rightarrow E_2\setminus E_1$ such that for every edge $e\in E_1\setminus E_2$, the edges $e$ and $\pi(e)$ are adjacent, and the third edge of the triangle formed by $e$ and $\pi(e)$ is in $E_1\cap E_2$.
Then a sequence of (single) edge slide operations can successively replace every edge $e\in E_1\setminus E_2$ with $\pi(e)$.
\end{lemma}
\begin{proof}
Note that each pair $(e,\pi(e))\in (E_1\setminus E_2)\times(E_2\setminus E_1)$ satisfies the geometric condition of the edge slide operation. Therefore, it remains to prove that for every subset $\widehat{E}\subset E_1\setminus E_2$, the replacement of the edges in $\widehat{E}$ with their images $\pi(\widehat{E})=\{\pi(e):e\in \widehat{E}\}$ produces a noncrossing straight-line spanning tree, that is, the graph $\widehat{G}=(S,(E_1\setminus \widehat{E})\cup \pi(\widehat{E}))$ is in $\mathcal{T}(S)$.
Note first that $E_1\cup E_2$ are pairwise noncrossing, and so $\widehat{G}$ is a noncrossing straight-line graph.
By construction, $\widehat{G}$ has $n-1$ edges, so it remains to show that $\widehat{G}$ is connected.
Suppose, to the contrary, that $\widehat{G}$ is disconnected, with connected components $\widehat{G}_1,\ldots, \widehat{G}_k$ for some $k\geq 2$. Since $T_1$ is a spanning tree, there is an edge $e\in \widehat{E}$ that joins two different components, say $\widehat{G}_1$ and $\widehat{G}_2$, where $1\leq i<j\leq k$.
Assume $e=uv$ and $\pi(e)=uw$, and without loss of generality, $u$ is in $\widehat{G}_i$ and $v$ is in $\widehat{G}_j$. By assumption, $vw\in E_1\cap E_2$, hence it is an edge in $\hat{G}$. Consequently, $w$ is in $\widehat{G}_j$, and so $\pi(e)$ connects two distinct components of $\widehat{G}$. Since $\pi(e)\in \pi(\widehat{E})$, then $\pi(e)$ is an edge of $\widehat{G}$, which contradicts the assumption that $\widehat{G}$ is disconnected.
\end{proof}

\begin{propo}\label{pro:sim-slide}
For every $n\geq 3$, there exist a set $S$ of $n$ points in general position and two trees in $\mathcal{T}(S)$
such that $\Omega(n)$ simultaneous edge slides are required to transform one into the other;
that is, $f_{\ses}(n)=\Omega(n)$.
\end{propo}
\begin{proof}
Aichholzer and Reinhardt~\cite{aichholzer2007quadratic} constructed a set $S$ of $n\geq 3$ points
in general position and two trees $T_1,T_2\in \mathcal{T}(S)$ such that $\Omega(n^2)$ simultaneous edge slides
are required to transform $T_1$ into $T_2$.
Consider a sequence of $m$ simultaneous edge slides that transforms $T_1$ into $T_2$.
If an edge $e_1=pq$ slides into $e_2=pr$, then edge $qr$ must be present before and after the operation,
and at most two edges can slide along $qr$ simultaneously (at most one on each side of $qr$).
Overall at most $2\lfloor \frac{n-1}{3}\rfloor$ edges can slide simultaneously.
By Lemma~\ref{lem:simulate}, we can perform the edge slides in each simultaneous operation sequentially, and obtain a sequence of at most $2m\cdot \lfloor \frac{n-1}{3}\rfloor$ edge slides that transform $T_1$ into $T_2$. The lower bound $\Omega(n^2)$ yields $m\geq \Omega(n)$, as claimed.
\end{proof}

\subsection{Convex Position}
\label{edgeslide:convex}

For single edge slide operations, the lower bound $\lfloor \frac{3n}{2}\rfloor-5$ follows from the corresponding bound for stronger operations.
Theorem~\ref{thm:slide_upper_convex} below, building on Lemmas~\ref{Lemma: P to P'} and~\ref{lem:sum},
provides a linear upper bound.

\begin{lemma} \label{Lemma: P to P'}
Given a set $S$ of $n\geq 3$ points in convex position and two paths $P_1$ and $P_2$
that each consist of edges of $\conv(S)$, we can transform $P_1$ into $P_2$ using $n-2$ edge slides.
\end{lemma}
\begin{proof}
Note that $P_{1}$ and $P_{2}$ differ at most by a single edge pair. Label the vertices clockwise from $v_1$ to $v_n$, such that $v_1v_n$ is an edge in $P_1$, but not an edge in $P_2$, and let $v_kv_{k+1}$ be the edge in $P_2$ that is not an edge in $P_1$.

Starting with $P_1$, we can use a sequence of $k-1$ successive slides moving $v_{i-1}v_n$ to $v_iv_n$, for $i=2,3,\ldots , k$, effectively replacing $v_1v_n$ with $v_kv_n$. Similarly, a sequence of $n-k-1$ edge slides moving $v_kv_{j+1}$ to $v_kv_{j}$, for $j=n-1,n-2,\ldots, k+1$, replaces $v_kv_n$ with $v_kv_{k+1}$. The concatenation of these two sequences transforms $P_1$ into $P_2$ using $(k-1)+(n-k-1)=n-2$ edge slides. See \figurename~\ref{fig:Path_To_Path} for an example.
\end{proof}

\begin{figure}[htbp]
\centering
\includegraphics{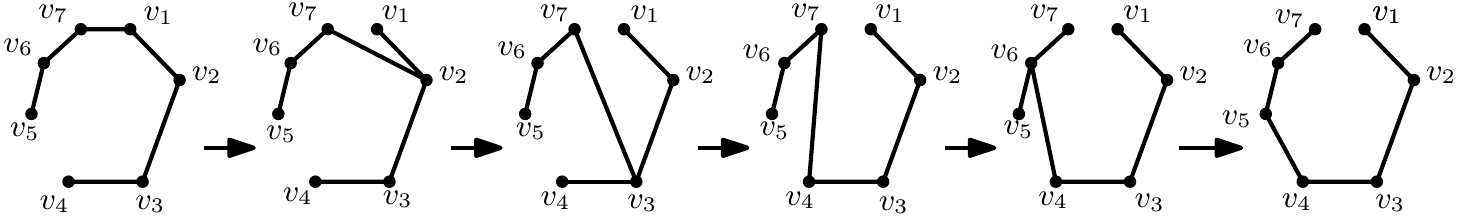}
\caption{An example for $n=7$ and $k=4$, where path $P_1$ is transformed into path $P_2$ by applying $n-2=5$ edge slides.}
\label{fig:Path_To_Path}
\end{figure}

\begin{lemma}\label{lem:sum}
Let $S$ be a set of $n\geq 3$ points in convex position and let $T\in \mathcal{T}(S)$. Assume that the dual tree of $T$ has $t$ cells $C_0,\ldots , C_{t-1}$, where $n_i$ is the number of points of $S$ on the boundary of $C_i$ for $i=0,\ldots, t-1$. Then
\begin{equation}\label{eq:sum}
n-2=\sum_{i=0}^{t-1}(n_i-2).
\end{equation}
\end{lemma}
\begin{proof}
Without loss of generality, let $C_0$ be the root of the dual tree. Then every cell $C_i$, $i>0$, shares precisely two vertices with its parent. Counting the number of new points in $S$ encountered by the BFS, we obtain  $n=n_0+\sum_{i=1}^{t-1}(n_i-2)$, which yields~\eqref{eq:sum}.
\end{proof}

\begin{theorem}\label{thm:slide_upper_convex}
For every set $S$ of $n$ points in convex position, every plane tree in $\mathcal{T}(S)$ can be transformed into any other tree in $\mathcal{T}(S)$ using at most $2n-5$ edge slides;
that is, $f_{\es}^{\cx}(n) \leq 2n-5$.
\end{theorem}
\begin{proof}
We show that any two plane spanning trees $T_1,T_2\in \mathcal{T}(S)$ on a set $S$ of $n\geq 3$ points in convex position can be transformed into the same path $P$ using a sequence of at most $n-3$ and $n-2$ edge slides,
respectively.

Consider the dual tree of $T_1$ and choose any cell $C_0$ of the dual tree to be the root.
Denote by $e_0$ the hull edge of $C_0$ (i.e., the edge of $C_0$ that is not in $T_1$), and let
$P\in \mathcal{T}(S)$ be the path formed by the remaining $n-1$ edges of $\conv(S)$.
If the dual tree has only one node, then $T_1=P$.
Otherwise, denote the remaining cells by $C_1,\ldots ,C_{t-1}$ in a BFS traversal of the dual tree.
For $i=0,\ldots ,t-1$, let $n_i$ be the number of points of $S$ on the boundary of $C_i$.

While the dual tree has two or more nodes, we merge $C_0$ with one of its children as follows.
Let $C_i$ be a child of $C_0$ in the dual tree and let $e_i$ be the edge shared by $C_0$ and $C_i$.
Note that $C_i$ induces a path in $T_1$ that consists of edges of the cell $C_i$.
We can apply Lemma~\ref{Lemma: P to P'} for the $n_i$ points on the boundary of $C_i$
to transform $e_i$ into the hull edge of cell $C_i$ using $n_i-2$ edge slides.
As a result, cells $C_0$ and $C_i$ merge to one cell. We let this cell be the new root cell. See \figurename~\ref{fig:Convex_es} for an example.

\begin{figure}[htbp]
\centering
\includegraphics{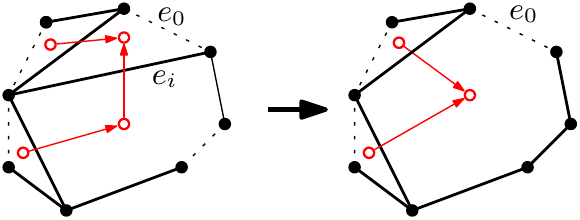}
\caption{One can merge the two cells which share edge $e_i$ by using 3 edge slides.}
\label{fig:Convex_es}
\end{figure}

When the while loop terminates, all cells $C_1,\ldots ,C_{t-1}$ have been merged into the root.
At this time, the dual tree has only one node, and the tree has been transformed into $P$.
By Lemmas~\ref{Lemma: P to P'} and~\ref{lem:sum}, we have used $\sum_{i=1}^{t-1}(n_i-2)=n-n_0\leq n-3$ edge slides.

If edge $e_0$ is absent from $T_2$, we can transform $T_2$ into $P$ as described above using $n-3$ edge slides.
However, if $e_0$ is an edge of $T_2$, we first apply an edge slide to replace $e_0$ with some other edge,
followed by a sequence of $n-3$ edge slides to obtain $P$.
The total number of operations is at most $2n-5$, as claimed.
\end{proof}

Now, let us consider simultaneous edge slides. We start with an easy lower bound.

\begin{theorem}\label{thm:sim_slide_lower_convex}
For every $n\geq 3$, there exist two trees in $\mathcal{T}(S)$, where $S$ is a set of $n$ points in convex position, that require $\Omega(\log n)$ simultaneous edge slides to transform one into the other;
that is, $f_{\ses}^{\cx}(n)=\Omega(\log n)$.
\end{theorem}
\begin{proof}
Consider two different paths $P_1$ and $P_2$ along the convex hull of a point set of size $n$ in convex position where the edge $uv$ is in $P_2$ but not in $P_1$. Let $(P_1=T_0,T_1,\ldots , T_m=P_2)$ be a sequence of trees in $\mathcal{T}(S)$ such that any two consecutive trees are related by a simultaneous edge slide.
For $i=0,\ldots , m$, let $C(i)$ be the cell that is incident to $uv$.
Since $T_0=P_1$ is a path along the convex hull, $C(0)$ is the only cell in the dual graph, in particular it is incident to all $n$ vertices. To transform $P_1$ into $P_2$, the edge $uv$ needs to be added;
thus, one has to slide edges of $P_1$ until cell $C(i)$ vanishes (i.e., its size drops to 2).
The size of $C(i)$ decreases only if an edge of $C(i)$ slides along another edge of $C(i)$, that is,
any size-decreasing edge slide involves two consecutive edges of $C(i)$.
Consequently, a simultaneous edge slide decreases the size of $C(i)$ by at most a factor of 2,
and so any sequence of simultaneous edge slides must use at least $m\geq \log_2 (n/2)=\Omega(\log n)$ operations.
\end{proof}

In the proof of the following result, we repeatedly apply a reduction step that ``removes'' a constant fraction of the leaves; this idea was originally developed for simultaneous flip operations in triangulations~\cite{Bose2007,Galtier2013}.

\begin{theorem}\label{thm:sim_slide_upper_convex}
Every plane tree in $\mathcal{T}(S)$, where $S$ is a set of $n$ points in convex position,
can be transformed into any other tree in $\mathcal{T}(S)$ using $O(\log n)$ simultaneous edge slides;
that is, $f_{\ses}^{\cx}(n)=O(\log n)$.
\end{theorem}
\begin{proof}
Let $S$ be a set of $n\geq 3$ points in convex position, let $p\in S$ and $T_1\in \mathcal{T}(S)$.
It is sufficient to show that $T_1$ can be transformed into a star centered at $p$ using $O(\log n)$ simultaneous edge slides.

The outline of the proof is as follows.
We transform $T_1$ into a star centered at $p$ via some intermediate phases where each phase uses $O(\log n)$ simultaneous edge slide operations.
The assumption that $S$ is in convex position is crucial for maintaining the planarity of the intermediate trees.
In Phase~1, we transform $T_1$ into a spanning tree $T_2$ in which every dual cell has $O(1)$ vertices.
In Phase~2, we transform $T_2$ into a tree $T_3$ of diameter $O(\log n)$.
Finally in Phase~3, $T_3$ is transformed into a star centered at $p$.
All these phases require $O(\log n)$ simultaneous edge slides.

\paragraph{Phase~1: Constant-size cells.}
Let $pq$ be an arbitrary convex hull edge. Define the cell incident to $pq$ to be the root of the dual tree,
and direct all edges of the dual tree toward the root (cf.~\figurename~\ref{fig:dualtree}(b)).
The edge of $T_1$ that separates a cell $C$ from its parent is called the \emph{parent edge} of $C$.
The hull edge and the parent edge split the boundary of $C$ into two paths (a 1-vertex path is possible).
Note that the edges of these paths lie on the boundary of the convex hull of their vertex set; we call such a path a \emph{convex path}. A simultaneous edge slide can decrease the length of a convex path from $k\geq 2$
to $\lceil k/2\rceil$ by sliding every other edge along the previous edge of the path.
Note that these slides can be performed simultaneously in a cell $C$.
If $e$ is an edge in a convex path of $C$ and $e$ is incident to another cell~$C'$,
then $e$ is the parent edge of $C'$ and therefore there is no slide in $C'$ that involves~$e$,
however, an edge slide performed on the convex path of $C$ may insert one new edge into $C'$.

We apply simultaneous edge slides, each modifying all convex paths of length two or higher,
while there is a cell with 7 or more vertices.
If a cell has $m$ vertices, $m\geq 3$, then the two convex paths on its boundary jointly have $m-2$ edges.
The two convex paths lose at least $\lfloor (m-3)/2\rfloor$ vertices, and the cell may gain at most one vertex
from its parent. In particular, a cell with $m\in \{3,4\}$ vertices may gain at most one vertex;
a cell with $m\in\{5,6\}$ vertices loses at least one vertex and gains at most one; and
the size of a cell with $m\geq 7$ vertices strictly decreases. The size of every cell with $m\geq 7$ vertices
goes down by a factor of $(m-\lfloor (m-3)/2\rfloor+1)/m = \lceil (m+5)/2\rceil/m\leq\frac{7}{8}$ or less.
Therefore the while loop terminates after $O(\log n)$ simultaneous edge slides.
We obtain a tree $T_2$ where every dual cell has 6 or fewer vertices.

Denote by $t$ the number of nodes of the dual tree of $T_2$; and let $n_0,\ldots , n_{t-1}$
denote the number of vertices of the $t$ cells. We derive a lower bound on $t$ by double counting the
number of incident cell-vertex pairs $\sum_{i=0}^{t-1}n_i$. On one hand, $n_i\leq 6$ for $i=0,\ldots , t-1$,
hence $\sum_{i=0}^{t-1}n_i\leq 6t$. On the other hand, Lemma~\ref{lem:sum} yields $\sum_{i=0}^{t-1}n_i=n+2(t-1)$.
The inequality $n+2t-2\leq 6t$ yields $t \geq (n-2)/4$.

\paragraph{Phase~2: Creating good leaves.}
A leaf of a spanning tree $T$ (or of a subtree) is called \emph{good}
if the edge incident to it is an edge of the convex hull of the vertices of $T$.
Note that if we remove a good leaf from $T$ to obtain a tree $T'$, then edge slides on the resulting tree $T'$ can also be performed in the entire tree $T$ (that is, the edge of a good leaf does not obstruct any edge slide in $T'$).
The main idea of transforming $T_2$ into a tree $T_3$ of diameter $O(\log n)$ is
to repeatedly make a constant fraction of vertices to be good leaves and then
``remove'' them (meaning that these leaves are disregarded in later iterations).

Let $T_2$ be a spanning tree with cells of size at most six.
Let $t_0$, $t_1$, and $t_2$, denote the number of nodes of the dual
tree with 0, 1, and more children, respectively, where $t_0 + t_1 + t_2 = t$.
Note that $t_0 \geq t_2$. We show that we can apply $O(1)$ simultaneous edge
slides until a constant fraction of the surviving vertices are good leaves
of the current tree.

First we perform the following ``clean-up'' step.
Let $\mathcal{C}_0$ (resp., $\mathcal{C}_1$) be the nodes in the dual tree that have precisely one child and
are at even (resp., odd) distance from the root. We assign the cells to $\mathcal{C}_0$ and $\mathcal{C}_1$ at the beginning of this phase, and do not reassign them later even though their distance from the root may change.
If a cell corresponding to a node in $\mathcal{C}_0$ and the cell of its child jointly have at most four vertices,
we transform them into a single cell by sliding the edge between them
into the hull edge of the parent cell using an edge slide (see \figurename~\ref{fig:convex_slide_degree_2} for two examples).
These edge slides can be performed simultaneously.
If a cell $C_1\in \mathcal{C}_1$ is merged with its parent cell $C_0\in \mathcal{C}_0$,
we replace $C_1$ with the combined cell $C_0\cup C_1$ in the set $\mathcal{C}_1$.
All other cells in $\mathcal{C}_1$ remain in $\mathcal{C}_1$ (even though their distance from the root might change).
Note that every cell in $\mathcal{C}_1$ still has precisely one child.
Then, we perform an analogous transformation for every cell in $\mathcal{C}_1$.
Both clean-up operations maintain the invariant that every cell has at most 6 vertices.
The two clean up operations take two simultaneous edge slides; denote by $T_2'\in \mathcal{T}(S)$
the resulting tree.

We claim that if a node in the dual graph of $T_2'$ has precisely one child,
then the corresponding cell and the cell of its child jointly have 5 or more vertices.
Suppose, to the contrary, that cells $C_a$ and $C_b$ are parent and child,
with four vertices combined, and $C_a$ has only one child. Then both $C_a$ and $C_b$
are triangles, hence they have not been merged with any other cell during the clean-up steps.
This implies that $C_a\in \mathcal{C}_0\cup \mathcal{C}_1$, and $C_a$ would
have been merged with $C_b$ in one of the two clean-up steps, contradicting our
assumption that both $C_a$ and $C_b$ are triangles.

\begin{figure}[htbp]
\centering
\includegraphics[page=1]{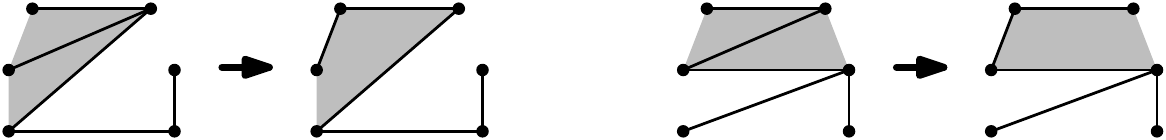}
\caption{Triangular cells of degree 2 are merged by one slide only involving the parent edge.}\label{fig:convex_slide_degree_2}
\end{figure}

Next, while there is a cell $C_a$ with a single child $C_b$ and a single grandchild $C_c$, do the following.
The cells $C_a$ and $C_b$ each have at most 6 vertices, and they jointly have at least 5 vertices
(due to the clean-up step). By Lemma~\ref{Lemma: P to P'}, we can transform the edge between
$C_a$ and $C_b$ into the hull edge of $C_a$ or $C_b$ using at most $6-2=4$ edge slides, as described below.
The parent edge of $C_a$ and the parent edge of $C_c$ split the boundary of $C_a\cup C_b$
into two paths $P_1$ and $P_2$ (the paths $P_1$ and $P_2$ need not be contained in the current tree),
which jointly have at least $5-2=3$ edges, including the hull edges of $C_a$ and $C_b$;
see \figurename~\ref{fig:convex_slide_degree_merge} for examples.
Without loss of generality, assume that $P_1$ has two or more edges and
contains the hull edge of $C_a$ or $C_b$.
If $P_1$ contains the hull edge of both $C_a$ and $C_b$, then we slide
the edge between $C_a$ and $C_b$ into the hull edge of $C_a$;
otherwise we slide it to the hull edge in $P_2$.
As a result $P_1$ contains the hull edge of the combined cell $C_a\cup C_b$.
An endpoint of this hull edge is in the interior of $P_1$,
and hence it is a good leaf of the current tree.

\begin{figure}[htbp]
\centering
\includegraphics[page=2]{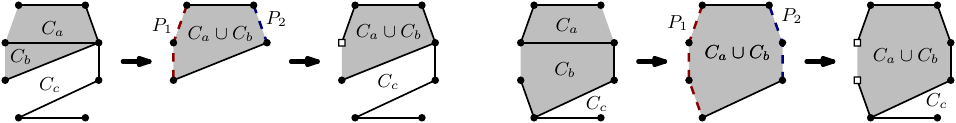}
\caption{For cells $C_a$ with a single child $C_b$ and a single grandchild $C_c$ such that $C_a$ and $C_b$ jointly have at least five vertices, we can slide the edge between $C_a$ and $C_b$ to the boundary of $\conv(S)$ and obtain at least one good leaf.}
\label{fig:convex_slide_degree_merge}
\end{figure}

In the dual tree of $T_2'$, consider the maximal paths induced by the $t_1$ nodes with a unique child;
note that these paths are vertex disjoint. Denote by $p$ the number of these paths. It is easy to show
that $p\leq t_0+t_2-1$. Indeed, if we compress each path into their unique child of the lowest node,
then we obtain a rooted tree with $t_0+t_2$ nodes, where each path is compressed into a unique node,
but not into the root. A maximum matching $M$ in these paths covers all but at most one vertex in each path.
That is, $|M|\geq (t_1-p)/2\geq (t_1-(t_0+t_2-1))/2=(t_1-t_0-t_2+1)/2$.
For each pair in $M$, we can create one good leaf, as described above.

Let us count the number of good leaves we can obtain with these operations.
For each of the $t_0$ leaves of the dual tree of $T_2'$, and
for each pair of nodes in the matching $M$, we obtain at least one good leaf.
Thus the number of good leaves is at least $t_0+(t_1-t_0-t_2+1)/2=(t_1+t_0-t_2+1)/2$.

We argue that we obtain at least $t/6$ good leaves. This certainly holds if $t_0 \geq t/6$.
Suppose now that $t_0 < t/6$. Then $t_0 \geq t_2$ yields $t_0+t_2 < t/3$ hence $t_1 > 2t/3$;
and further $(t_1+t_0-t_2+1)/2\geq (t_1+1)/2> (2t/3+1)/2>t/6$, as required.
Using the bound $t\geq (n-2)/4$ obtained above, this yields at least $(n-2)/24$ good leaves.

We can now summarize the steps to transform $T_2$ into $T_3$ using $O(\log n)$ simultaneous edge slide operations.
Starting from $T_2$, we repeatedly create good leaves and remove them.
Denote by $L_i$ the set of good leaves removed in iteration $i$.
Each iteration removes at least a $\frac{1}{24}$-fraction of the vertices.
In each iteration, the edges incident to good leaves are edges of the convex hull of the
current subtree, consequently at most two such edges are incident to the same vertex in the subtree.
After $r \in O(\log n)$ iterations, we are left with a single vertex $p$.
The tree $T_3\in \mathcal{T}(S)$ is the tree obtained by these ``removal'' operations.
Specifically, the tree $T_3$ is rooted at $p$, and $L_i$ is the set of vertices
distance $i$ from $p$ for $i=1,\ldots ,\lceil\log_{24}n\rceil$. Consequently,
the depth (hence diameter) of $T_3$ is $O(\log n)$, as claimed.
As noted above, each vertex in $T_3$ has at most two children.

\begin{figure}[htbp]
\centering
\includegraphics{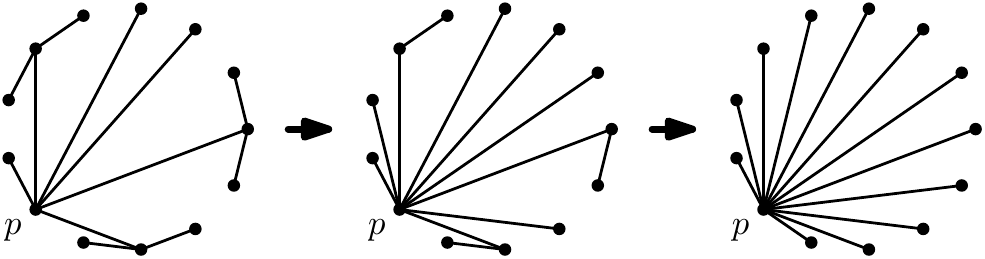}
\caption{One round of sliding good leaves of a subtree into a star centered at~$p$.}
\label{fig:slide_to_star}
\end{figure}

\paragraph{Phase~3: Creating a star.}
The one-vertex tree on $p$ is a star.
We re-insert the leaves in $L_i$ for $i=r,r-1,\ldots ,1$
(in reverse order) in $r$ rounds, and transform the subtree into a star centered at $p$.
In round $j$, we re-insert the edges to the vertices in $L_{r+1-j}$.
They are each adjacent to the current star centered at $p$,
and each vertex of the current star is incident to at most two edges in $L_{r+1-j}$, as noted above.
For each leaf of the current star, we can slide an incident edge in $L_{r+1-j}$ to
the root, and these edge slides can be performed simultaneously.
Using up to two simultaneous edge slides, all the edges in $L_{r+1-j}$ become incident to $p$;
see~\figurename~\ref{fig:slide_to_star} for an illustration.
After $r$ rounds, we obtain the star centered at $p$.
As mentioned earlier, the edge slides performed in a subtree
can also be performed in the whole tree, as the edges incident
to good leaves do not obstruct any edge slides.
This completes the proof.
\end{proof}

\section{Labeled Edges}
\label{sec:labeled}
In this section, we prove new lower and upper bounds for the diameters of the transition graphs in the edge-labeled setting, which are
summarized in Tables~\ref{tbl:labeled_general} and~\ref{tbl:labeled_convex}.

\subsection{Exchange and Compatible Exchange}

\subsubsection{General Position}

Recall that one simultaneous exchange can transform any (unlabeled) tree $T=(S,E_1)$ into any other tree $T_2=(S,E_2)$,
since it is enough to move the edges in $E_1\setminus E_2$ to $E_2\setminus E_1$. In the edge-labeled case, however, the edges in $E_1\cap E_2$ may have different labels in the two trees. For example, if $E_1=E_2$ and the labels are shifted cyclically, then one simultaneous exchange cannot move any edge in $E_1$ to its position in $E_2$. Two simultaneous exchanges suffice if there exists a tree $T_3=(S,E_3)$ such that $E_1\cap E_3=\emptyset$. This strategy does not work directly for simultaneous compatible exchanges: Garc\'{\i}a et al.~\cite{Garcia14} constructed a tree $T_1\in \mathcal{T}(S)$ such that any compatible tree $T_2\in \mathcal{T}(S)$ has at least
$(n-2)/5$ edges in common with $T_1$. We start with an easy observation about unlabeled spanning trees.

\begin{propo}\label{pro:disjoint-trees}
Let $S$ be a set of $n\geq 2$ points in general position.
\begin{enumerate}\itemsep 0pt
\item\label{disjoint:1} If $S$ is not in convex position, then there exist two compatible edge-disjoint trees in $\mathcal{T}(S)$.
\item\label{disjoint:2} If $S$ is in convex position, then for every edge $uv$ of $\conv(S)$, there exist two compatible trees
      in $\mathcal{T}(S)$, namely a path along $\conv(S)$ and a star centered at $u$ or $v$,
      such that both contain edge $uv$, but they do not share any other edges.
\end{enumerate}
\end{propo}
\begin{proof}
Assume that $S$ is not in convex position, hence $n\geq 4$. Let $v_1\in S$ be an arbitrary point in the interior of $\conv(S)$,
and label the remaining points in $S$ by $v_2,\ldots , v_n$ in radial order around $v_1$. Since $v_1$ lies in the interior of $\conv(S)$, we have $\angle v_i v_1 v_{i+1}<\pi$ for $i=1,\ldots, n-1$; and $\angle v_n v_1 v_2<\pi$. In particular, the spanning star centered at $v_1$ and the cycle $(v_1,v_2,\ldots , v_n)$ are compatible. We can now define two compatible edge-disjoint trees $T_A,T_B\in \mathcal{T}(S)$. The tree $T_A$ is obtained from the spanning star centered at $v_1$ by replacing the edge $v_1 v_2$ with $v_2 v_3$. The tree $T_B$ is obtained from the cycle $(v_2,\ldots , v_n)$ by replacing the edge $v_2 v_3$ with $v_1 v_2$.

Assume now that $S$ is in convex position, and its vertices are labeled $v_1,\ldots , v_n$ in counterclockwise order around the boundary of $\conv(S)$. Let $T_A$ be the spanning path $(v_1,\ldots , v_n)$, and let $T_B$ be the spanning star centered at $v_1$. It is clear that $T_A$ and $T_B$ are compatible, and only the edge $v_1v_2$ appears in both trees.
\end{proof}

\begin{propo}\label{pro:labeled-sim-exchange}
If $S$ is not in convex position, then every edge-labeled plane tree in $\mathcal{L}(S)$ can be transformed into any other tree in $\mathcal{L}(S)$ using at most 3 simultaneous exchanges, or $O(\log n)$ simultaneous compatible exchanges.
\end{propo}
\begin{proof}
Let $T_1,T_2\in \mathcal{L}(S)$ be two plane trees. By Proposition~\ref{pro:disjoint-trees}\eqref{disjoint:1}, there are two \emph{unlabeled} edge-disjoint compatible plane trees $T_A,T_B\in \mathcal{T}(S)$. We can now describe the transformation between the edge-labeled plane trees $T_1$ and $T_2$. Ignoring the labels, transform $T_1$ into $T_A$ using one simultaneous exchange (resp., $O(\log n)$ simultaneous compatible exchanges~\cite{aichholzer2002sequences}). Similarly, transform $T_2$ into $T_B$. These operations produce some edge labelings on $T_A$ and $T_B$, respectively. Since $T_A$ and $T_B$ are compatible and edge-disjoint, one simultaneous compatible exchange transforms $T_A$ into $T_B$ with matching labels.
\end{proof}

Now, we show the upper bound of 3 in Proposition~\ref{pro:labeled-sim-exchange} is tight.

\begin{propo}\label{pro:labeled-sim-lower}
Let $T_1,T_2\in \mathcal{L}(S)$, $|S|\geq 3$ be two edge-labeled spanning stars with the same center. If each (unlabeled) edge has two different labels in the two stars, then 3 simultaneous exchanges are required to transform $T_1$ into $T_2$.
\end{propo}
\begin{proof}
Let $v\in S$ be the common center of the spanning stars $T_1$ and $T_2$. Suppose that 2 simultaneous exchanges can transform $T_1$ into $T_2$. The first operation carries $T_1$ into some tree $T_3$, in which $v$ is incident to at least one edge, say $uv$, and hence this operation did not move $uv$. Therefore, the second operation moves $uv$ to another edge, which implies that $uv$ is not present in $T_2$, contradicting the assumption that $T_2$ is a star centered at~$v$.
\end{proof}

\subsubsection{Convex Position}

If $S$ is in convex position, then $\mathcal{T}(S)$ does not contain two edge-disjoint spanning trees, and we need a more careful analysis for handling edge labels.

\begin{propo}\label{pro:labeled-cx-sim-exchange}
Every edge-labeled plane tree in $\mathcal{L}(S)$, for a set $S$ in convex position, can be transformed into any other tree in $\mathcal{L}(S)$ using at most 3 simultaneous exchanges, or 4 simultaneous compatible exchanges.
\end{propo}
\begin{proof}
Let $T_1,T_2\in \mathcal{L}(S)$ be two plane trees. If $|S|\in \{1,2,3\}$, then it is easily checked that $T_1$ can be transformed into $T_2$ using at most two simultaneous exchanges. Assume $|S|\geq 4$. We distinguish between three cases.

\paragraph{Case~1. An edge $uv$ of $\conv(S)$ is present with the same label in both $T_1$ and $T_2$.}
In this case, we proceed similarly to the proof of Proposition~\ref{pro:labeled-sim-exchange}.
Transform $T_1$ into a spanning path $T_A$ along the boundary of $\conv(S)$ that contains edge $uv$, with one endpoint at $u$; this takes one simultaneous (compatible) exchange. Transform $T_2$ into a spanning star $T_B$ centered at $u$. This takes one simultaneous exchange; and up to two simultaneous compatible exchanges: the first transforms $T_1$ into a path on the boundary of $\conv(S)$ that contains $uv$, and the second transforms this path into $T_B$.
Since $T_A$ and $T_B$ are compatible and share only edge $uv$, which has the same label in both trees,
one simultaneous compatible exchange can transform $T_A$ into $T_B$.

\paragraph{Case~2. $T_1$ and $T_2$ each correspond to some unlabeled path along the boundary of $\conv(S)$, but none of the edges of $\conv(S)$ is in both $T_1$ and $T_2$ with the same label.}
Note that $T_1$ and $T_2$ are each compatible with every tree in $\mathcal{T}(S)$.
Assume that $T_1$ corresponds to the unlabeled path $(v_1,\ldots , v_n)$ along $\conv(S)$; see \figurename~\ref{fig:Case2} for an example.
We may further assume that edge $v_1v_2$ in $T_1$ has the same label as edge $v_iv_{i+1}$ in $T_2$ for some $i\in \{2,\ldots , n-1\}$.
We apply three simultaneous exchanges. First, transform $T_1$ into a star $T_A$ centered at $v_1$ using one simultaneous (compatible) exchange. Note that edge $v_1v_2$ remains fixed. Since no other edge remains fixed, we may assume that
if edge $v_1v_n$ is present in $T_2$, then the edge of $T_1$ with the same label is mapped to $v_1v_n$; and
if $v_2v_3$ is present in $T_2$ and $i>2$, then the edge of $T_1$ with the same label is mapped to $v_1v_{i+1}$.

\begin{figure}[htbp]
\centering
\includegraphics{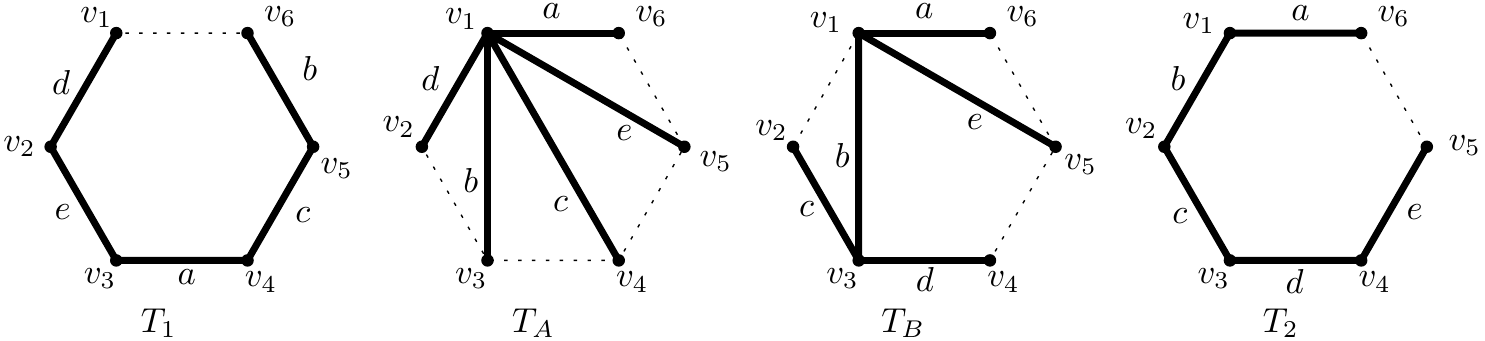}
\caption{Three compatible exchanges in Case~2, where $n=6$ and $i=3$.}
\label{fig:Case2}
\end{figure}

The graph $T_A+v_iv_{i+1}$ contains a unique cycle, namely the 3-cycle $(v_1,v_i,v_{i+1})$.
Let $T_B=T_A+v_2v_3-v_1v_2$ if $i=2$; and $T_B=T_A+v_iv_{i+1}-v_1v_{i+1}+v_2v_3-v_1v_2$ if $i>2$.
We can transform $T_A$ into $T_B$ with one simultaneous (compatible) exchange that maps $v_1v_2$ to $v_iv_{i+1}$,
and $v_1v_{i+1}$ to $v_2v_3$ (if $i>2$). Note that $T_B$ has up to 3 edges on the boundary of $\conv(S)$
(namely $v_2v_3$, $v_iv_{i+1}$, and $v_1v_n$), and these edges have the same label in both $T_B$ and $T_2$ if they are present in $T_2$.
Therefore, we can transform $T_B$ into $T_2$ using one simultaneous (compatible) exchange.

\paragraph{Case~3. No edge of $\conv(S)$ is present with the same label in both $T_1$ and $T_2$; and $T_1$ or $T_2$ has an edge in the interior of $\conv(S)$.} Without loss of generality, assume that $T_1$ has an edge $e_1$ in the interior of $\conv(S)$.
The tree $T_2$ cannot contain all edges of $\conv(S)$, since they form a cycle.
Without loss of generality, we may assume that $v_1v_2$ is not present in $T_2$.
Since $e_1$ is in the interior of $\conv(S)$, at least two edges of $\conv(S)$, say $h_1$ and $h_2$, are not present in $T_1$.
Transform $T_1$ into a path $T_A=(v_1,\ldots ,v_n)$ along the boundary of $\conv(S)$, using one simultaneous (compatible) exchange.
In particular, if $v_1v_n$ is present in $T_1$, then we can transform $e_1$ and $v_1v_n$ to $h_1$ and $h_2$, respectively.
Let $e_2$ be the edge in $T_2$ that has the same label as $v_1v_2$ in $T_A$.
Transform $T_2$ into a star $T_B$ centered at $v_1$ such that edge $e_2$ is mapped to $v_1v_2$.
This takes one simultaneous exchange, or two simultaneous compatible exchanges: the first transforms $T_2$ into the
path $(v_1,\ldots v_n)$ along the boundary of $\conv(S)$, and the second transforms this path into the star centered at $v_1$.
Since $T_A$ and $T_B$ are compatible and share only edge $v_1v_2$, which has the same label in both trees,
one simultaneous compatible exchange can transform $T_A$ into $T_B$.

In all three cases, we have transformed $T_1$ into $T_2$ using 3 simultaneous exchanges (resp., up to 4 simultaneous compatible exchanges), as claimed.
\end{proof}

\subsection{Rotation}

We prove a linear upper bound for the diameter of $\mathcal{G}_{\ro}^L(S)$ for a set of $n$ points in general position.

\begin{theorem}\label{thm:labeled-rotation}
Every edge-labeled plane tree in $\mathcal{L}(S)$, where $|S|=n\geq 2$, can be transformed into any other
using at most $11n-22$ rotations.
\end{theorem}
\begin{proof}
Let $p_1$ be an extreme point in $S$, and let $p_2, \dots, p_n$ be the remaining vertices, indexed in counterclockwise order around $p_1$.
Ignoring the edge labels, we can transform any tree in $\mathcal{L}(S)$ into the same path $P=(p_1, \dots, p_n)$, using $2(n-2)$ rotations. Indeed, Avis and Fukuda~\cite{avis1996reverse} showed that every tree in $\mathcal{T}(S)$ can be transformed into a star centered at any extreme point of $S$ using at most $n-2$ rotations,
and the star centered at $p_1$ can be transformed into $P$ using $n-2$ additional rotations.
Every order of the edges along $P$ corresponds to a permutation of the $n-1$ labels.
Every permutation can be carried to any other permutation using at most $(n-1)-1=n-2$ transpositions.
It remains to show how such a transposition is implemented using a constant number of edge rotations.

Suppose we want to exchange the labels of the edges $p_i p_{i+1}$ and $p_j p_{j+1}$, with $1 \leq i < j < n$.
We rotate $p_ip_{i+1}$ to $p_1 p_{i+1}$ (we may omit this step if $i=1$).
Then, we rotate the resulting edge $p_1 p_{i+1}$ to $p_1 p_j$ (which may be omitted if $j=i+1$).
We can now exchange the labels of $p_j p_{j+1}$ and $p_1 p_j$ using three rotations, as shown in \figurename~\ref{fig:label change};
this operation temporarily adds the edge $p_1 p_{j+1}$, which is always possible as no part of the tree is in the interior of the triangle $\Delta(p_1 p_j p_{j+1})$. Once the edge $p_1 p_j$ has its new label, we rotate it back to its initial position $p_i p_{i+1}$, using the reverse of the rotation sequence described above (i.e., rotating it to $p_1 p_{i+1}$ and then back to $p_i p_{i+1}$).
As exchanging two labels requires not more than seven rotations, the total number of rotations to transform one labeled tree into another is at most $4(n-2) + 7(n-2) = 11(n-2)$.
\end{proof}

While we can transform a tree into a path using $O(\log n)$ simultaneous rotations by Theorem~\ref{thm:rotations}, we currently do not see how to change the permutation of the labels using a sub-linear number of simultaneous rotations.

\subsection{Empty-Triangle Rotation}

For both empty-triangle rotations and edge slides, we make use of the following operation used by Cano et al.~\cite[\figurename~9]{cano2013edge}. Two labeled edges that appear consecutively around a common vertex can be exchanged using three edge slides: If edges $up$ and $vp$ are consecutive in the radial order around $p$, and
$\angle upv\leq \pi$, we can exchange them by sliding $up$ to $uv$, sliding $vp$ to $up$, and then sliding $uv$ to $vp$; see \figurename~\ref{fig:label change}. Hence, when considering the order of labeled edges around a vertex as a permutation, an adjacent transposition in the permutation can be implemented by three edge slides.
\begin{figure}
\centering
\includegraphics{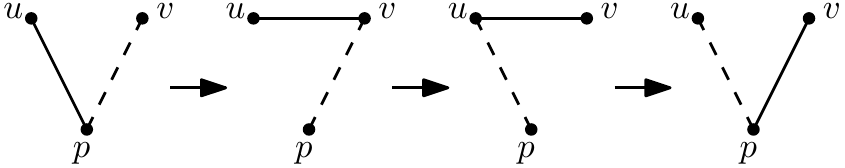}
\caption{Exchanging the labels of two adjacent edges of an empty triangle  using 3 edge slides.}
\label{fig:label change}
\end{figure}

\subsubsection{General Position}

For point sets in general position, we can establish an upper bound of $O(n \log n)$, which matches the current best bound for the unlabeled case.

\begin{theorem}\label{thm:labeled_empty_triangle_rotation_upper}
Every edge-labeled plane tree in $\mathcal{L}(S)$, $|S|=n$, can be transformed into any other
using $O(n \log n)$ empty-triangle rotations.
\end{theorem}
We use an idea that has previously been used to show that any point set in general position has a triangulation whose dual graph has its diameter in $O(\log n)$~\cite{dual_diameter}.
\begin{proof}
Let $\tilde T$ be a triangulation of a convex $n$-gon whose dual graph has a diameter of $O(\log n)$.
As every outerplanar graph with $n$ vertices admits a straight-line embedding on every set of $n$ points in general position~\cite{pgmp-eptvsp-91}, we may embed $\tilde T$ on~$S$. Let $T$ be this embedding.
Ignoring the labels, we can transform any tree in $\mathcal{L}(S)$ into the same path $P$ that is on the outer face of $T$ using $O(n\log n)$ empty-triangle rotations (cf.~Theorem~\ref{thm:empty_rotation_general}).
Let $e$ be the unique edge from the outer face of $T$ that is not in $P$.
Now let $a$ and $b$ be two edges of $P$.
We show how to exchange the labels of $a$ and $b$ using $O(\log n)$ edge slides.
In a path, the linear order of the edges corresponds to a permutation of the $n-1$ labels.
Since every permutation can be carried to any other permutation using at most $(n-1)-1=n-2$ transpositions,
we obtain an overall bound of $O(n \log n)$.

Let $t_a$ and $t_b$ be the triangles of $T$ that are incident to $a$ and $b$, respectively.
The dual graph of $T$ is a tree, which we root at the triangle incident to $e$.
Let $t$ be the lowest common ancestor of $t_a$ and $t_b$.
If $t = t_a = t_b$, we are done.
Otherwise, we have, say, $t_a \neq t$.
Then, we rotate $a$ to become another edge of $t_a$;
there is a unique choice for the new edge $a'$, and $a'$ is incident to a triangle that is closer to $e$ (and thus $t$) in the dual. We can thus iteratively bring the edges with the labels of $a$ and $b$ closer to the triangle $t$, until they are incident to it and then we can exchange them using three edge slides as indicated in \figurename~\ref{fig:label change}. As the diameter of the dual of $T$ is in $O(\log n)$, this process takes $O(\log n)$ rotations.
\end{proof}

For simultaneous empty-triangle rotations, we use a different approach. Similarly to the proof of Theorems~\ref{thm:labeled-rotation} and~\ref{thm:labeled_empty_triangle_rotation_upper}, the challenge is to permute the edge labels of a fixed tree in $\mathcal{L}(S)$, and implement the transpositions using empty triangle rotations. For simultaneous operations, we can use parallel sorting algorithms. The \emph{odd-even transposition sort} is a procedure of sorting $n$ elements in $O(n)$ rounds: In even rounds, each element with an even index is compared (and possibly swapped) with its successor, and in odd rounds this is done for each element with odd index (see, e.g.,~\cite[p.~721]{cormen}).

\begin{theorem}\label{thm:labeled_sim_empty_triangle_rotation_upper}
Every edge-labeled plane tree in $\mathcal{L}(S)$, $|S|=n$, can be transformed into any other
using $O(n)$ simultaneous empty-triangle rotations.
\end{theorem}
\begin{proof}
Let $p$ be an extreme point in $S$. Let $T_1,T_2\in \mathcal{L}(S)$. Ignoring the edge labels, we can transform both trees into the same star centered at $p$, using at most $8n$ simultaneous empty-triangle rotations by Theorem~\ref{thm:sim_empty_upper}.
A sequence of 3 edge slides can exchange two edges that are consecutive in the radial order around $p$, as shown in \figurename~\ref{fig:label change}. This corresponds to an adjacent transposition in a permutation. Edge slides in nonadjacent triangles can be performed simultaneously. We can sort the labels using odd-even transposition sort. Thus, in $O(n)$ rounds, each of which can be implemented using 3 simultaneous edge slides, the labels are sorted. Overall we use only $O(n)$ simultaneous empty-triangle rotations to transform one edge-labeled spanning tree into any other.
\end{proof}

\subsubsection{Convex Position}

\begin{propo}\label{pro:labeled-cx-emptytriangle}
Every edge-labeled plane tree in $\mathcal{L}(S)$, for a set $S$ of $n\geq 3$ points in convex position, can be transformed into any other using at most $6n-13$ empty-triangle rotations, or at most $O(\log n)$ simultaneous empty-triangle rotations.
\end{propo}
\begin{proof}
Let $T_1,T_2\in \mathcal{L}(S)$. Ignoring the labels, we can transform both trees into the same (unlabeled) path $P$ along the boundary of $\conv(S)$ using at most $2n-5$ empty-triangle rotations, as in the proof of Theorem~\ref{thm:slide_upper_convex}.
Two edges, say $uv$ and $u'v'$, on the boundary of $\conv(S)$ span an empty quadrilateral or triangle. They can be exchanged using at most 4 empty-triangle rotations as indicated in \figurename~\ref{fig:label change} (in an empty triangle), or \figurename~\ref{fig:label change2} (in an empty quadrilateral). An exchange corresponds to a transposition in the permutation of labels along $P$. Therefore $4(n-2)$ empty-triangle rotations can transform one labeled path into the other. Overall, we use at most $(2n-5)+4(n-2)=6n-13$ operations.

\begin{figure}
\centering
\includegraphics{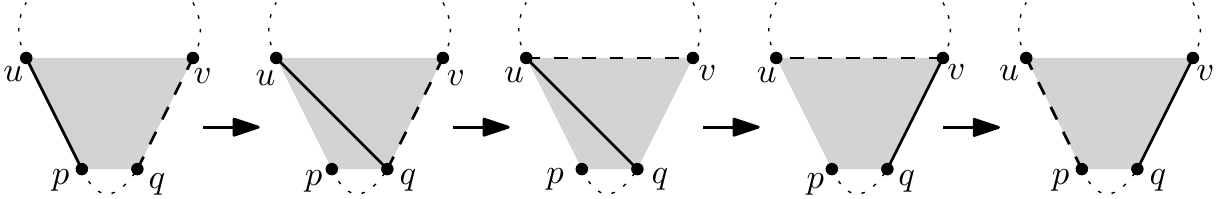}
\caption{Exchanging the labels of two nonadjacent edges of a path in convex position.}
\label{fig:label change2}
\end{figure}

\medskip
Ignoring the labels, we can transform both trees into the same (unlabeled) path~$P$ along the boundary of $\conv(S)$ using at most $4$ simultaneous empty-triangle rotations by Theorem~\ref{thm:sim_empty_upper_convex}.
We obtain two labeled paths $P_1,P_2\in \mathcal{L}(S)$. It remains to permute the labeled edges of the paths. We show below that this can be accomplished using $O(\log n)$ simultaneous empty-triangle rotations, with a variant of parallel quicksort. As noted above, each transposition can be implemented using at most 4 empty-triangle rotations. Importantly, we can perform transpositions simultaneously if the convex hulls of the exchanged edges are pairwise interior-disjoint.

We divide $P$ into two (unlabeled) paths, $P_{\rm left}$ and $P_{\rm right}$, of length $\lfloor (n-1)/2\rfloor$ and $\lceil n/2\rceil$, respectively. Let $M_1$ be the set containing every edge in $P_1$ that lies in $P_{\rm left}$ but the corresponding edge in $P_2$ with the same label is in $P_{\rm right}$. Similarly, let $M_2$ be the set of every edge in $P_1$ that lies in $P_{\rm right}$ but the edge in $P_2$ with the same label is in $P_{\rm left}$. Note that $|M_1|=|M_2|$. Consequently there is a perfect matching between $M_1$ and $M_2$; we choose a perfect matching such that
the $i$th edge from the left in $M_1$ is matched to the $i$th edge from the right in $M_2$, for $i=1\ldots, |M_1|$.
This ensures that the convex hulls of distinct pairs are interior-disjoint. Consequently, four simultaneous rotations can exchange the edges $M_1$ and $M_2$ in $P_1$. Recursion on the first and second halves of $P$, respectively, can be performed simultaneously, and so the sorting algorithm terminates after $O(\log n)$ iterations, and transforms $P_1$ into $P_2$.
\end{proof}

\subsection{Edge Slide}

\subsubsection{General Position}

\begin{theorem}\label{thm:labeled_edge_slides_upper}
Every edge-labeled plane tree in $\mathcal{L}(S)$, $|S|=n\geq 2$, can be transformed into any other
using $O(n^2)$ edge slides. 
\end{theorem}
\begin{proof}
Let $p$ be an extreme point in $S$.
Ignoring the edge labeling, we can transform both trees into the same star centered at $p$, using $O(n^2)$ edge slides~\cite{aichholzer2007quadratic}.
In a star, the radial order of the edges around $p$ corresponds to a permutation of the $n-1$ labels.
We sort this permutation using the adjacent transposition of \figurename~\ref{fig:label change}.
Every permutation can be carried to any other permutation using at most $O(n^2)$ adjacent transpositions (e.g., by the bubble sort algorithm). Overall, we use $O(n^2)$ edge slide operations.
The matching lower bound $\Omega(n^2)$ follows from the unlabeled version~\cite{aichholzer2007quadratic}.
\end{proof}

Note that if two trees are stars with the same center, we can apply $O(n)$ simultaneous edge slides to obtain an identical labeling; we can exchange neighboring labels in the manner of odd-even transposition sort, which is known to finish after $O(n)$ rounds. However, we do not know whether a star can be obtained using $o(n^2)$ simultaneous edge slides.

\subsubsection{Convex Position}

\begin{theorem}\label{thm:labeled-edgeslide-cx-upper}
Every edge-labeled plane tree in $\mathcal{L}(S)$, for a set $S$ of $n$ points in convex position, can be transformed into any other using $O(n\log n)$ edge slides, or $O(n)$ simultaneous edge slides.
\end{theorem}
\begin{proof}
We use $O(n)$ edge slides to transform both trees into a canonical one (defined below) by Theorem~\ref{thm:slide_upper_convex}, and then adjust the labels. We define a canonical tree as follows.
Let $r\in S$ be an arbitrary vertex of $\conv(S)$; $r$ will be the root of the canonical tree.
Let $\{p_1, \dots, p_{n-1}\}$ be the set of remaining vertices indexed in counterclockwise radial order around~$r$.
Add the edge $r p_{\ceil{(n-1)/2}}$. The supporting line of  $r p_{\ceil{(n-1)/2}}$ splits the set of remaining vertices into two subsets, $\{p_1, \ldots, p_{\ceil{(n-1)/2}}\}$ and $\{p_{\ceil{(n-1)/2}}, \ldots, p_{n-1}\}$.
In each subset, we designate $p_{\ceil{(n-1)/2}}$ as the root and recurse until all points are
connected to the tree. This recursive algorithm constructs a rooted binary tree of height $O(\log n)$
such that a vertex, its parent, and its grandparent form an empty triangle.
Thus, by three edge slides (cf.~\figurename~\ref{fig:label change}), we can exchange the
label of an edge between a vertex and its parent with the label of the edge between its
parent and grandparent.
Consequently, using $O(\log n)$ edge slides, we can exchange the labels of any two edges.
Overall, it takes $O(n \log n)$ edge slides to permute the labels, as claimed.

\medskip
For simultaneous edge slides, we can use an approach similar to Theorem~\ref{thm:labeled_empty_triangle_rotation_upper}.
We transform the two trees into the same star using $O(\log n)$ simultaneous edge slides by Theorem~\ref{thm:sim_slide_upper_convex}.
Then, we can swap the labels of two edges that are consecutive in the radial order around the center of the star
as shown in \figurename~\ref{fig:label change}.
As we can use edge slides to change the labels of consecutive edges, we can permute the labels as in odd-even transposition sort.
This results in the desired permutation of the labels after $O(n)$ rounds.
\end{proof}

We prove a lower bound of $\Omega(n\log n)$ using a technique developed by Sleator, Tarjan, and Thurston~\cite{STT92}, which was used for establishing a lower bound of $\Omega(n\log n)$ for the diameter of flip graphs over $n$-vertex triangulations~\cite{STT92}.
Bose et al.~\cite{BLPV18} extended the technique to derive the same lower bound for edge-labeled triangulations on $n$ points in convex position.
In a nutshell, Sleator, Tarjan, and Thurston~\cite{STT92} show that, given a graph with $n$ vertices and an operation that can replace any subgraph isomorphic to some connected graph of size $O(1)$ with some other subgraph of size $O(1)$, then $d$ successive operations can produce at most $2^{O(n+d)}$ distinct graphs. The same technique works for edge-labeled graphs. Since every plane tree with $n$ vertices has $(n-1)!=2^{\Theta(n\log n)}$ edge labelings, at least $d=\Omega(n \log n)$ operations are needed to reach all possible edge labelings.

We note that the transition graph $\mathcal{G}_{\es}^L(S)$, where $S$ is a set of $n$ points in convex position, is $(2n-4)$-uniform (see Lemma~\ref{lem:uniform} below). So the naive bound on the $d$-neighborhood of a node in $\mathcal{G}_{\es}^L(S)$ is $O(n^d)=2^{O(d\log n)}$, which would give only a trivial lower bound of $\diam(\mathcal{G}_{\es}^L(S))\geq \Omega(n)$.
The new insight is that many sequences of edge slides lead to the same output;
in particular, Lemma~\ref{lem:simulate} implies that if $k$ edge slides can be performed simultaneously, then they can also be performed sequentially in any of the $k!$ possible orders, producing the same output.

\begin{lemma}\label{lem:uniform}
Let $S$ be a set of $n\geq 3$ points in convex position.
Then the transition graph $\mathcal{G}_{\es}^L(S)$ is $2(n-2)$-uniform.
\end{lemma}
\begin{proof}
Let $P$ be a spanning path on the boundary of $\conv(S)$. Then an edge of $P$ can slide along another edge if and only if the two edges are adjacent. There are precisely $n-1$ edges and $n-2$ (unordered) pairs of adjacent edges in $P$. Hence there are $2(n-2)$ ordered pairs of adjacent edges $(uv,vw)$, which determine an edge slide from $uv$ to $uw$. All $2(n-2)$ possible edge slide operations produce distinct trees.

Let $S$ be a set of $n\geq 3$ points in convex position and $T\in \mathcal{T}$ be an arbitrary tree. Assume that the dual tree has $t$ nodes for some $1\leq t\leq n-3$. Denote the corresponding cells by $C_0,\ldots , C_{t-1}$, and let $n_i$ be the number of vertices of cell $C_i$ for $i\in \{0,\ldots ,t-1\}$. Every edge slide of $T$ is determined by an ordered pair of consecutive edges along the boundary of one of the cells. Any two cells share at most one edge, so they determine distinct pairs. The vertices in each cell induce a path, which yields $2(n_i-2)$ such pairs. By Lemma~\ref{lem:sum}, the overall number of edge slides is $\sum_{i=1}^{t-1}2(n_i-2)=2(n-2)$, as claimed.
\end{proof}

We review terminology from~\cite{BLPV18} and~\cite{STT92}. A \emph{half-edge} in a graph is an incident vertex-edge pair.
Let $\Delta$ be an absolute constant. A graph $G$ with $n$ vertices and maximum degree $\Delta$ is \emph{half-edge tagged}
if every half-edge of $G$ is labeled (\emph{tagged}) by an integer in $\{1,\ldots , \Delta\}$ such that
at each vertex the half-edges have distinct labels.\footnote{We call labels of the half-edges \emph{tags}, consistently with~\cite{STT92}, to distinguish them from the edge labels of trees in $\mathcal{S}$.}
Recall that the \emph{incidence graph} $I(G)$ of a graph $G=(V,E)$ is a bipartite graph where the partite sets correspond to $V$ and $E$, and an edge represents a vertex-edge incidence. An \emph{half-edge tagged graph part} is a set of vertices, edges, and half-edges of $G$ that correspond to a connected subgraph in $I(G)$ such that for every vertex in $V$, it contains all incident half-edges (but it need not contain both half-edges of an edge).

Local modifications of a half-edge tagged graph are represented by a so-called graph grammar.
A \emph{graph grammar} $\Gamma$ is a finite set of production rules $\{L_i \rightarrow_i R_i\}$, where the $i$th production rule comprises the \emph{left side} $L_i$, the \emph{right side} $R_i$ and the \emph{correspondence} $\rightarrow_i$.
Every left side $L_i$ and right side $R_i$ are incidence-labeled graph parts that have the same number
of mismatched half-edges. The correspondence $\rightarrow_i$ bijectively maps the half-edges of $L_i$ to those of
$R_i$ (however, the corresponding half-edges may have different tags).
Such a production $L_i \rightarrow_i R_i$ applies to a graph $G$ if $I(G)$ contains
a subgraph isomorphic to $I(L_i)$, including the half-edge labels.
We apply the production by replacing an occurrence of $L_i$ with $R_i$ accordingly.
A \emph{derivation} is a sequence of graphs $G = G_0, G_1, \ldots, G_m = G'$ such that each
$G_i$ is obtained from $G_{i-1}$ by applying a production rule.

\begin{theorem}[{Sleator, Tarjan, and Thurston~\cite[Theorem~2.3]{STT92}}]\label{thm:short_encodings}
Let $G$ be a graph of $n$ vertices, $\Gamma$ be a graph grammar,
$c$ be the total number of vertices in all left sides of $\Gamma$,
and $r$ be the maximum number of vertices in any right side of $\Gamma$.
Let $R(G, \Gamma, m)$ be the set of graphs obtainable from $G$ by derivations in $\Gamma$ of length at most~$m$.
Then $|R(G, \Gamma, m)| \leq (c+1)^{n+rm}$.
\end{theorem}

The above theorem continues to hold if each vertex of $G$ has a label, as pointed out in~\cite[Section~3.4]{STT92}.
We cannot apply the machinery of~\cite{STT92} to the plane spanning trees in $\mathcal{L}(S)$ directly, since they
have unbounded degree. We apply it for a variant of the line graphs defined as follows.

Let $S$ be a set of $n\geq 3$ points in convex position. For every plane spanning tree $T=(S,E)\in \mathcal{T}(S)$, we define a \emph{reduced line graph} $G_T$ in the following way.
Recall the definition of cells in Section~\ref{convex and er}.
\begin{itemize}
\item The vertices of $G_T$ are the edges of $T$.
\item Let $e$ and $f$ be two edges of $T$. In $G_T$, there is an edge between the vertices $e$ and $f$ if and only if $e$ and $f$ are adjacent edges along the boundary of some dual cell of~$T$.
\end{itemize}
By definition, $G_T$ is a subgraph of the line graph of $T$.
Note that $G_T$ has precisely $n-1$ vertices, and, by Lemma~\ref{lem:uniform}, precisely $n-2$ edges.
The maximum degree of a vertex is $4$, as any edge in $T$ is incident to two vertices and at most two cells.
When $T\in \mathcal{L}(S)$ is an edge-labeled spanning tree, we label the vertices of $G_T$ with the edge labels of~$T$.

We also define \emph{valid} half-edge tags for $G_T$ as follows. Let $\Delta=4$. For every edge $e$ of $T$, let $C(e)$ be the union of (up to two) dual cells whose boundary contains $e$. Note if $e$ and $f$ are adjacent, then $f$ is an edge of $C(e)$.
In the reduced line graph $G_T$, we label the half-edges of $e$ by integers $1,\ldots, \deg_{G_T}(e)$ in counterclockwise
order of the neighbors of $e$ along $C(e)$; see \figurename~\ref{fig:slide_production} for examples.
Thus each vertex $e$ of $G_T$ has $\deg_{G_T}(e)\leq \Delta$ valid half-edge tags.

In general, we cannot reconstruct the tree $T$ from its half-edge tagged reduced line graph $G_T$, since neither $G_T$ nor the half-edge tags contain enough information to identify the points in $S$. Nevertheless, we show that at most $2n$ trees in $\mathcal{L}(S)$ produce a given half-edge tagged reduced line graph.

\begin{lemma}\label{lem:reconstruct}
Let $S$ be a vertex set of a regular $n$-gon for $n\geq 3$. For every edge-labeled tree $T=(S,E)\in \mathcal{L}(S)$, we can reconstruct $T$ from the half-edge tagged reduced line graph $G_T$ up to rotations and reflections of $S$. Consequently, $\mathcal{L}(S)$ produces at least $|\mathcal{L}|/(2n)$ distinct half-edge tagged reduced line graphs.
\end{lemma}
\begin{proof}
We proceed by induction on $n$. In the base case ($n=3$), every tree in $\mathcal{L}(S)$ is an edge-labeled path of length 2 along $\conv(S)$, which are equivalent under rotations and reflections of $S$.

For the induction step, assume that $n\geq 4$ and the claim holds for $n-1$ in place of $n$.
We claim that $T$ has an edge $pq$ such that $p$ is a leaf of $T$ and $pq$ is an edge of $\conv(S)$.
Indeed, let $C$ be the cell that corresponds to a leaf of the dual tree; hence it is adjacent to at most one other cell.
Let $P$ be the path in $T$ induced by vertices on the boundary of $C$. The first or last edge of $P$
is not on the boundary of any other cell that has two leaves, hence it is an edge of $\conv(S)$,
and one of its endpoints is a leaf of both $P$ and $T$. This completes the proof of the claim.

By construction, $pq$ is a leaf of $G_T$; we denote its neighbor by $f$ (which is an adjacent edge of the path $P$ above).
After deleting $pq$ from $G_T$, and updating the half-edge tags of $f$ such that they are consecutive integers, we obtain a valid half-edge tagged reduced line graph $G_{T'}$ of the edge-labeled spanning tree $T'=(S\setminus \{p\},E\setminus \{pq\})$. By induction we can reconstruct $T'$ from $G_{T'}$ up to rotations and reflections. The half-edge labels at $f$ uniquely determine the position of point $p$ relative to $q$.
\end{proof}

\begin{theorem}\label{thm:labeled-edgeslide-cx-lower}
Let $S$ be a set of $n\geq 3$ points in convex position, and let $T\in \mathcal{L}(S)$.
There exists an edge-labeled tree $T'\in \mathcal{L}(S)$ such that
transforming $T$ into $T'$ requires at least $\Omega(n\log n)$ edge slides,
and at least $\Omega(\log n)$ simultaneous edge slides.
\end{theorem}
\begin{proof}
Every edge-labeled tree $T\in \mathcal{L}(S)$ determines a vertex-labeled and half-edge tagged reduced line graphs $G_T$ as defined above.
Let $\mathcal{H}(S)=\{G_T: T\in \mathcal{L}(S)\}$ be the set of these graphs. We now define a graph grammar $\Gamma$ on $\mathcal{H}(S)$ such that each edge slide in $\mathcal{L}(S)$ corresponds to a production in $\mathcal{H}(S)$ (not necessarily the other way round).

\begin{figure}[htbp]
\centering
\includegraphics{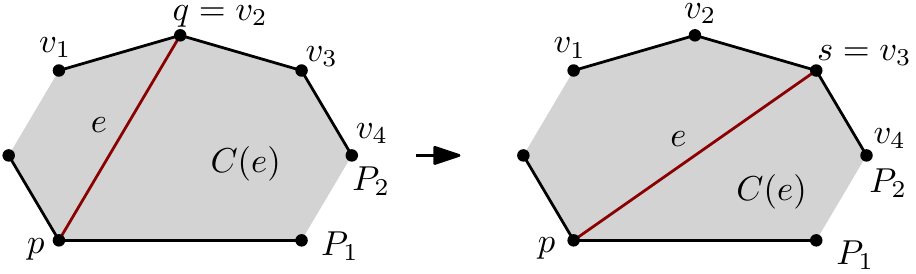}
\caption{An edge slide replaces $pq$ with $ps$, where $p\in P_1$ and $q,s\in P_2=(v_1,\ldots, v_4)$.}
\label{fig:production}
\end{figure}

\begin{figure}[htbp]
\centering
\includegraphics{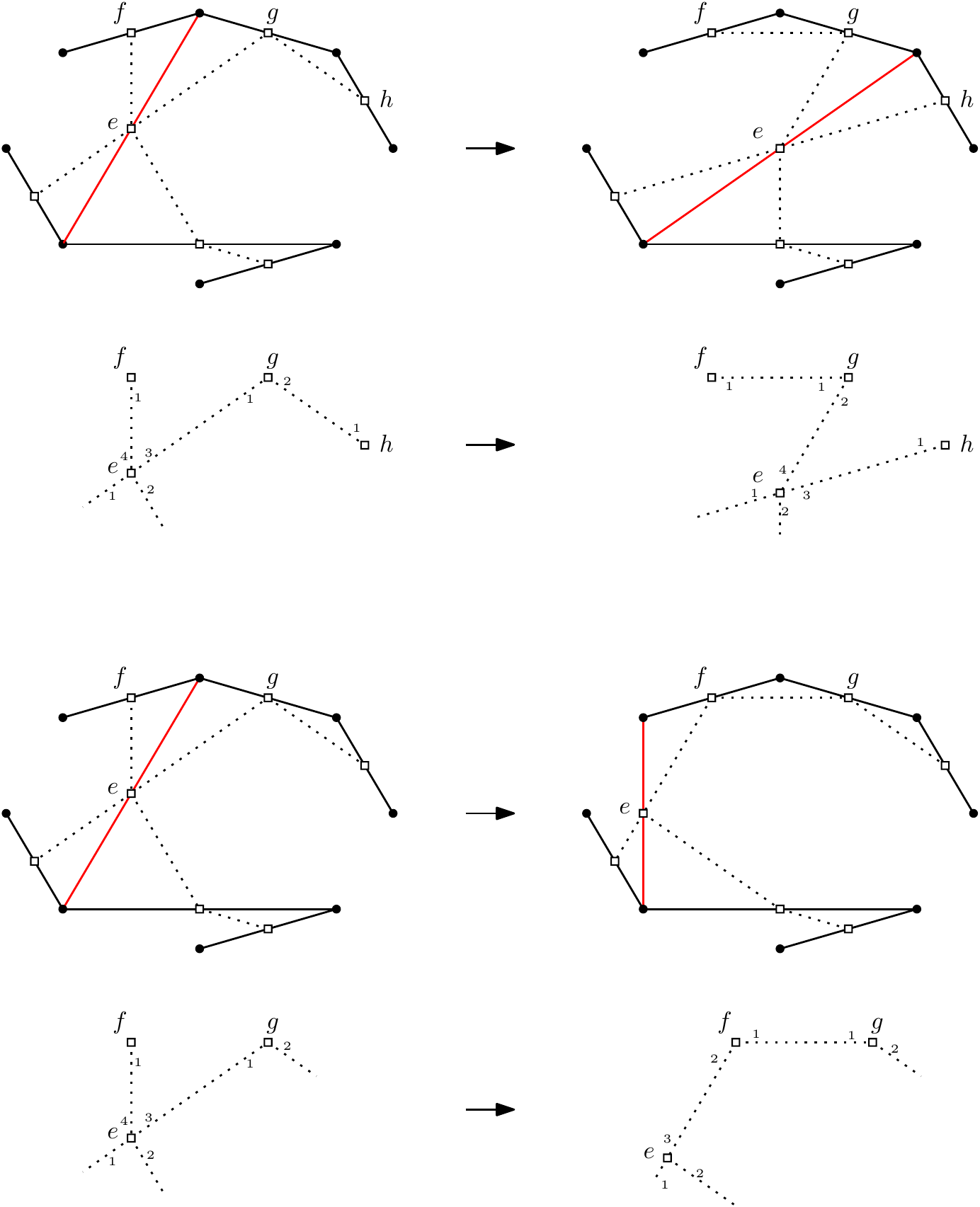}
\caption{Two edge slides and the corresponding production in the reduced line graph (dotted with square vertices). In the production at the top, the half-edges are mapped according to their labels, but in the bottom production, there is one exception: the half-edge $(e,4)$ is mapped to $(f,2)$. As the number of edges in each reduced line graph is $n-2$,
there is such a mapping for every edge slide and thus a corresponding production in the edge graph.}
\label{fig:slide_production}
\end{figure}

Let $T\in \mathcal{L}(S)$, and suppose that an edge slide operation replaces $pq$ with $ps$, and transfers the edge label $e$ from $pq$ to $ps$; see \figurename~\ref{fig:production}. Let $C(pq)$ denote the union of the two dual cells whose boundary contains $pq$. Then the boundary of $C(pq)$ contains two paths in $T$, that we denote by $P_1$ and $P_2$ such that $p\in P_1$ and $q\in P_2$. Since $pq$ and $qs$ are on the boundary of some dual cell, $qs$ is an edge of $P_2$. Assume, without loss of generality, that $P_2=(v_1,\ldots, v_k)$ for some $k\geq 3$, with $q=v_j$ and $s=v_{j+1}$. In the reduced line graph $G_T$, the impact of the edge slide is limited to $e$ and any edge on the boundary of $C(pq)$ incident to $p$, $q$, or $s$. However, the edges incident to $p$ in $P_1$ continue to be adjacent to $e$ in $G_T$. So the impact of the edge slide is limited to $e$ and up to 3 edges in $P_2$ incident to $q$ or $s$; that is, at most four edges.

Ignoring the half-edge tags, there are four cases depending on whether $q=v_1$ and whether $s=v_k$, respectively, each defines a production rule; in each case the right side $R_i$ and the left side $L_i$ have at most 4 vertices. As $G_T$ has $n-2$ edges for every tree $T\in \mathcal{L}(S)$, there exists a bijection between the half-edges in $L_i$ and $R_i$ in all four cases.
Two examples of such productions are given in \figurename~\ref{fig:slide_production}.
Together with all possible valid half-edge tags for the (up to 4) vertices of $R_i$ and $L_i$, resp., the number of production rules is still bounded by a constant $c\leq 4\cdot 4\cdot (3!)^4$.

Therefore, there exist constants $c\in O(1)$ and $r=4$ such that
$c$ is the total number of vertices in all left sides of $\Gamma$,
and $r$ is the maximum number of vertices in any right side of $\Gamma$.

Let $m$ be the diameter of the transformation graph $\mathcal{G}^L_{\es}(S)$
for edge-labeled plane spanning trees under edge slides.
Then every edge-labeled tree in $\mathcal{L}(S)$ is within distance $m$ from $T$.
Hence every half-edge tagged reduced line graphs in $\mathcal{H}(S)$ can
be obtained from $G_T$ using at most $m$ production rules.
On one hand, $|\mathcal{H}(S)|\leq (c+1)^{n+rm}$ by Theorem~\ref{thm:short_encodings}.
On the other hand, $|\mathcal{H}(S)|\geq |\mathcal{L}(S)|/(2n)$ by Lemma~\ref{lem:reconstruct}.
Every spanning tree $T\in \mathcal{L}(S)$ has $n-1$ edges, hence $(n-1)!$ edge labelings,
which yields the lower bound $|\mathcal{L}(S)|\geq (n-1)!$.
It follows that $|\mathcal{H}(S)|\geq |\mathcal{L}(S)|/(2n)\geq (n-1)!/(2n)$.
By contrasting the lower and upper bounds for $|\mathcal{H}(S)|$, we obtain
\begin{eqnarray*}
(n-1)!/(2n) &\leq& (c+1)^{n+rm}\\
\log_{c+1}[(n-1)!/(2n)] &\leq& n+rm \\
n\log_{c+1} n -O(n) &\leq & rm\\
\Omega(n\log n) &\leq & m,
\end{eqnarray*}
as claimed. For simultaneous edge slide operations, a lower bound of $\Omega(\log n)$ follows
analogously to the proof of Proposition~\ref{pro:sim-slide}.
\end{proof}

\section{Conclusions}
\label{sec:conclusion}

Previous work introduced five elementary operations on the space of plane spanning trees $\mathcal{T}(S)$ on a point set $S$ in Euclidean space.
All five operations are known to define a connected transition graph.
This is the first comprehensive analysis of the diameters of these graphs.
Obvious open problems are to close the gaps between the lower and upper bounds in Tables~\ref{tbl:bounds_general}--\ref{tbl:labeled_convex}. 
One might also consider new variations.
For example, we obtain a new variant of \emph{empty-triangle rotation} if we require $\Delta(pqr)$ to be empty of vertices (but not necessarily edges), or a new variant of \emph{edge slide} when not requiring $\Delta(pqr)$ to be empty.
These variations have not been considered and may lead to new geometric insights.

Transition graphs on other common plane geometric graphs have been considered in the literature, but they do not allow for such a rich variety of operations. For the space of noncrossing matchings on $S$, a compatible exchange operation has been defined, but the transition graph is disconnected even if $S$ is in convex position~\cite{AAT15}; it is known that the transition graph has no isolated vertices~\cite{IST13}. Connectedness is known for bipartite geometric matchings, with a tight linear diameter bound~\cite{ABLS15}. For noncrossing Hamiltonian cycles (a.k.a.\ polygonizations) it is a longstanding open problem whether the transition graph of simultaneous compatible exchange is connected.

While our upper bounds on the diameters of transition graphs are constructive, the problem of determining the transformation distance between two given trees seems to be still open (or is trivial) in all settings we discussed. Similar problems have been studied for triangulations: it is NP-hard to determine the flip distance of two triangulations of a point set~\cite{lubiw_pathak,pilz}, but the problem is fixed-parameter tractable in their distance~\cite{triang_fpt}. Even though the transition graph has a small diameter for simultaneous operations, the maximum degree may be exponential in $|S|$, and the distance between two trees thus does not seem to be a suitable parameter for the complexity of the problem.
For points in convex position, the complexity of the related problem on triangulations (already posed in~\cite{STT88}) is still open.

Apart from the edge-labeled variant discussed in Section~\ref{sec:labeled}, we could also consider the problem on directed plane spanning trees.
In a directed spanning tree $T=(S,E)$, the direction of an edge $e_1\in E$ defines an order between the two components of $(S,E\setminus\{e\})$, and we can direct a replacement edge $e_2$ between the two components consistently with this order. Studying whether the transition graphs of directed plane spanning trees are connected, and estimating their diameters under various
operations, is left for future work. We note that if the edges are both labeled and directed, the corresponding exchange graph is no longer connected: for 3 vertices, the transition graph is 2-uniform and has several components.

\paragraph{Acknowledgment.}
Key ideas for our results on simultaneous edge slides were discussed at the GWOP~2017 workshop in Pochtenalp, Switzerland.
We thank all participants for the constructive atmosphere.
Research by Nichols and T\'oth was partially supported by the NSF awards CCF-1422311 and CCF-1423615.
Pilz is supported by a Schr\"odinger fellowship of the Austrian Science Fund (FWF): J-3847-N35; part of this work was done while he was at the Department of Computer Science of ETH Z\"urich.

Last but not least, we thank the anonymous reviewers for many helpful comments and suggestions that helped improve the clarity of presentation in this paper.

\bibliographystyle{plain}
\bibliography{Edge}
\end{document}